\documentclass[a4paper,11pt,oneside  ]{article}	

\makeindex
\usepackage{hyperref}
\usepackage{graphicx}
\usepackage{enumerate}
\usepackage[disable]{todonotes}
\providecommand{\todoneukamm}[1]{\todo[inline,author={SN}, color=yellow]{#1}}

\hypersetup{colorlinks, linkcolor=blue, urlcolor=red, filecolor=green, citecolor=blue}
\usepackage{amsmath}
\usepackage{amsthm}
\usepackage{amssymb}
\usepackage{authblk}
\usepackage{enumerate}
\usepackage{fullpage}
\usepackage{esint}
\usepackage{mathrsfs}

\theoremstyle{plain}
\newtheorem{theorem}{Theorem}[section]

\newtheorem{lemma}[theorem]{Lemma}
\newtheorem{proposition}[theorem]{Proposition}
\newtheorem{algorithm}[theorem]{Algorithm}

\theoremstyle{definition}
\newtheorem{definition}[theorem]{Definition}

\newtheorem{example}{Example}[section]

\theoremstyle{remark}

\newcommand{\N}{\mathbb{N}}
\newcommand{\R}{\mathbb{R}}

\newcommand{\frankconstant}{\boldsymbol{\kappa}}
\newcommand{\eps}{\ensuremath{\varepsilon}}

\newcommand\dist{\operatorname{dist}}
\newcommand\sym{\operatorname{sym}}
\newcommand{\SO}[1]{\operatorname{SO}(#1)}

\newcommand\id{{\operatorname{I}}}

\newcommand{\step}[1]{\medskip\noindent\textbf{Step #1. }}

\newcommand{\wto}{\rightharpoonup}

\newcommand{\dev}{{\rm dev}}
\def\barx{\bar{\boldsymbol x}}
\def\Hrel{\mathbb H_{\rm rel}}
\def\Hmicro{\mathbb H_{\rm micro}}
\def\Hmacro{\mathbb H_{\rm macro}}
\def\Hres{\mathbb H_{\rm res}}
\def\res{\operatorname{\rm res}}
\def\pre{\operatorname{\rm pre}}
\def\trace{\operatorname{\rm trace}}
\def\Skew{\operatorname{\rm skew}}

\begin{document}
\begin{center}
  \LARGE
  Modeling and simulation of nematic LCE rods
  \bigskip

  \normalsize
  S\"oren Bartels\footnote{bartels@mathematik.uni-freiburg.de}\textsuperscript{*}, Max Griehl\footnote{max.griehl@tu-dresden.de}\textsuperscript{\$}, Jakob Keck\footnote{jakob.keck@math.uni-freiburg.de}\textsuperscript{*}, and
  Stefan Neukamm\footnote{stefan.neukamm@tu-dresden.de}\textsuperscript{\$}
  \par \bigskip

  \textsuperscript{*}Department of Applied Mathematics, University of Freiburg \par
  \textsuperscript{\$}Faculty of Mathematics, Technische Universit\"at Dresden\par \bigskip

  \today
\end{center}

\begin{abstract}
  We introduce a nonlinear, one-dimensional bending-twisting model for an inextensible bi-rod that is composed of a nematic liquid crystal elastomer. The model combines an elastic energy that is quadratic in curvature and torsion with a Frank-Oseen energy for the liquid crystal elastomer. Moreover, the model features a nematic-elastic coupling that relates the crystalline orientation with a spontaneous bending-twisting term. We show that the model can be derived as a $\Gamma$-limit from three-dimensional nonlinear elasticity. Moreover, we introduce a numerical scheme to compute critical points of the bending-twisting model via a discrete gradient flow. We present various numerical simulations that illustrate the rich mechanical behavior predicted by the model. The numerical experiments reveal good stability and approximation properties of the method.


  \textbf{MSC-2020:} 74B20; 76A15; 74K10; 65N30; 74-10.
  
\end{abstract}
\tableofcontents
\section{Introduction}
Nematic liquid crystal elastomers (LCE) are nonlinearly elastic materials that feature an additional orientational internal degree of freedom. They are composed of a polymer network with incorporated liquid crystals---rod-like molecules that tend to align in a nematic phase. Nematic LCE feature a nematic-elastic coupling between the orientation of the liquid cyrstals and the elastic properties of the polymer network. More specificially, the material tends to stretch in the direction parallel to the liquid crystal orientation and shrinks in orthogonal directions. Furthermore, it is possible to tweak the orientation of the liquid crystals by fields and to (de)activate the nematic-elastic coupling by external stimuli (e.g., temperature). For this reasons,  nematic LCE exhibit interesting  mechanical properties (e.g., thermo-mechanical coupling \cite{B_broer11, B_Warner07} or soft elasticity \cite{B_Finkelman91, B_Kundler95}), and are used to design active materials, see \cite{B_White15,ware2016localized, MaJaZa18}. In this context, slender structure such as thin films and rods are of interest. In recent years nonlinear models for thin films and rods made of nematic LCE have also been intensively studied from a mathematical perspective. In particular, the derivation of lower-dimensional models from thee-dimensional models has been discussed, e.g., see \cite{B_conti2002soft, B_conti2002semisoft, B_Plucinsky18, B_plucinsky2018patterning, B_cesana2015effective,B_Agostiniani17plate, B_Agostiniani17platesoft, B_Agostiniani17platehetero, B_Agostiniani17ribbon, bauer2019derivation}, and reliable numerical schemes have been developed, e.g., see~\cite{BaPa21, bonito2021numerical, BBMN18, BaBoNo17, sander2010geodesic, sander2016numerical, le1992finite, bergou2008discrete}
for numerical methods for slender structures,
\cite{NoWaZh17, nitschke2020liquid, nochetto2017finite, borthagaray2020structure, walker2020finite}
for LCE and director fields,
and \cite{bartels2022nonlinear}
for models that combination thin-film mechanics and director fields.

\smallskip

Our paper is devoted to the modeling and simulation of nonlinear bi-rods that are composed of nematic LCE and a usual nonlinear elastic material. Starting from a three-dimensional nonlinear model that invokes a standard energy functional from nonlinear elasticity, a Frank-Oseen energy for the LCE material, and a nematic-elastic coupling term, 
we derive via $\Gamma$-convergence a model for an inextensible rod that is capable to describe large bending and twisting deformations, and a coupling mechanism that relates the local orientation of the LCE with a spontaneous curvature/torsion-term. In contrast to other works, e.g., \cite{goriely2022rod}, in our model the director field is not prescribed but an additional degree of freedom of the model; moreover, the derivation is ansatz-free in the sense of $\Gamma$-convergence. It is based on the one hand on the recently introduced approach in \cite{bartels2022nonlinear} where the derivation of nematic LCE plates is studied, and on the other hand, on \cite{neukamm2012rigorous, bauer2019derivation} where the simultaneous homogenization and dimension reduction of prestrained rods is analyzed as an extension of the seminal work \cite{mora2002derivation}.
\smallskip

Our effective one-dimensional description of LCE rods allows for
efficient numerical simulations of complex problem settings. We follow
\cite{BarRei20} and use standard $C^0$ and $C^1$ conforming finite
element spaces that are subordinated to a partitioning of the straight
reference configuration to approximate director fields and the bending
deformation, respectively. Geometric constraints such as the
conditions for the material frame are imposed at nodes of the
partitioning via suitable linearizations and penalty terms. Similarly,
the discretization of the unit-length constraint for the director
field describing the LCE orientation is imposed at the nodes. We then
use a semi-implicit gradient descent method to decrease the
one-dimensional energy functional from a given initial state to obtain
stationary configurations via sequences of linear problems with simple
system matrices. Our experiments show how a periodic bending behaviour
can be controlled in a quasi-static setting via a time-dependent external field.
Our simulations also illustrate how internal material parameters, that can be changed via external
stimuli, affect the bending behavior in the case of compressive,
twist-inducing boundary conditions. The numerical experiments reveal
good stability and approximation properties of the iterative method
and the devised discretization scheme. They lead to meaningful results
within minutes for an elementary implementation on standard desktop
computers.

\smallskip

The paper is structured as follows. In Section~\ref{S:2} we first introduce the one-dimensional bending-twisting model. We then introduce a three-dimensional nonlinear elasticity model and show that it $\Gamma$-converges to the one-dimensional model, see Section~\ref{S:three}. The effective coefficients of the one-dimensional depend on the original material, the geometry of the cross-section and the domain occupied by the LCE. In Section~\ref{S:4} we present their definition and derive simplified  formulas that hold in special settings. In Section~\ref{S:num} we introduce a discretization of the one-dimensional model minimization problem via a discrete gradient flow, and we explore the model via numerical simulations. All proofs are presented in Section~\ref{S:proof}.

\smallskip

\section{A bi-rod model for  nematic LCEs and its derivation}\label{S:2}
In this section we introduce a one-dimensional bending-twisting model that describes an inextensible rod that is composed of a nematic LCE and a nonlinearly elastic material, and we explain its derivation from three-dimensional nonlinear elasticity via $\Gamma$-convergence, see Section~\ref{S:three} below.

\subsection{The one-dimensional model}
We denote by $\omega:=(0,L)$ the reference domain of the rod. We describe a configuration of the rod by a triplet $(y,R,n)$ satisfying
\begin{equation}\label{def:rodconf}
  \begin{aligned}
    &y\in H^2(\omega;\R^3),\qquad R\in H^1(\omega;\R^{3\times 3}),\qquad n\in H^1(\omega;\R^3)\\
    &\text{such that for a.e. $x_1\in\omega$, }\partial_1y(x_1)=R(x_1)e_1,\quad R(x_1)\in\SO 3,\quad |n(x_1)|=1.
  \end{aligned}
\end{equation}
Here, $y$ describes the deformation of the rod and $R$ the deformation of an associated orthonormal frame. The field $n$ describes the orientation of the liquid crystals in global coordinates. We denote by
\begin{equation*}
  \mathcal A:=\{(y,R,n)\,:\,(y,R,n)\text{ satisfies }\eqref{def:rodconf}\,\}
\end{equation*}
the set of all \textit{rod configurations} and call a pair $(y,R)$ satisfying \eqref{def:rodconf} a \textit{framed curve}. We consider the energy functional $\mathcal E:\mathcal A\to\R$ defined by
\begin{align}\label{def:limitI}
  \begin{aligned}
    \mathcal E(y,R,n):=&\int_\omega  \bar Q\Big(R^\top\partial_1R +\bar r K_{\pre}(\tfrac13 I-\hat n\otimes \hat n)\Big)+\bar r^2 E_{\res}\big(\tfrac13 I-\hat n\otimes \hat n)\big)\,dx_1\\
    &+\frankconstant{}^2\int_\omega|\partial_1n|^2\,dx',\qquad \text{where }\hat n:=R^\top n,
  \end{aligned}
\end{align}
where
\begin{itemize}
\item $\bar Q:\R^{3\times 3}_{\Skew}\to\R$ is a positive definite quadratic form that describes the  bending--twisting energy; here and below, $\R^{3\times 3}_{\Skew}$ denotes the space of skew-symmetric matrices in $\R^{3\times 3}$.
\item $K_{\pre}:\R^{3\times 3}_{\dev}\to\R^{3\times 3}_{\Skew}$ is a linear map that describes the contribution of the nematic-elastic coupling that leads to spontaneous bending and twisting of the rod; here and below, $\R^{3\times 3}_{\dev}$ denotes the space of symmetric matrices in $\R^{3\times 3}$ with vanishing trace.
\item $E_{\res}:\R^{3\times 3}_{\dev}\to\R$ is a positive, semi-definite quadratic form that describes a residual energy that cannot be accomodated by bending or twisting of the rod.
\item $\bar r\in\R$ is a model parameter related to the strength of the nematic-elastic coupling.
\item $\frankconstant{}>0$ is a model parameter that is related to the scaling of the physical diameter of the rod, the shear moduls of the elastomer and the Frank elastic constant of the nematic LCE.
\end{itemize}

As we shall explain next, this model can be derived as a zero-thickness $\Gamma$-limit from a three-dimensional rod composed of an elastic material and an LCE-material. In this context, the precise definition of $\bar Q$ and $K_{\rm pre}$ depends on the considered material, the geometry of the cross-section of the rod and the geometry of the subdomain that is occupied by the nematic LCE material, see Definition~\ref{D:effective}.
\subsection{The three-dimensional model and $\Gamma$-convergence}\label{S:three}
The starting point of the derivation is the following three-dimensional situation: We consider an elastic composite material that occupies the three-dimensional, rod-like  domain $\Omega_h:=\omega\times hS$, where $0<h\ll 1$ denotes a (non-dimensionalized) thickness of the rod, $S\subset\R^2$ the rescaled cross-section, and $\omega:=(0,L)$ the mide-line of the rod. We assume that
\begin{equation}\label{ass:S}
  \begin{aligned}
    &S\subset\R^2\text{ is a simply connected, Lipschitz domain,}\\
    &\text{satisfying }\int_Sx_2\,d\bar x=\int_Sx_3\,d\bar x=\int_Sx_2x_3\,d\bar x=0.
  \end{aligned}
\end{equation}
Here and below, we use the notation $x=(x_1,\bar x)\in\R^3$ and $\bar x=(x_1,x_2)\in\R^2$. We note that the symmetry condition in \eqref{ass:S} is not a restriction, since it can always be achieved by rotating and translating $S$. We assume that the rod is composed of a conventional elastic material and a nematic LCE. The latter occupies the subbody $\Omega_{0,h}:=\omega\times (hS_0)$ where $S_0\subset S$ denotes a subdomain of the cross-section. We assume that
\begin{equation}\label{ass:S0}
  S_0\subset S\text{ is a simply connected, Lipschitz domain such that $|S_0|>0$ and $|S\setminus S_0|>0$}.
\end{equation}
To study the limit $h\to 0$, it is convenient to work with the rescaled domain $\Omega:=\omega\times S$ (resp. the rescaled subdomain $\Omega_0:=\omega\times S_0$). We therefore describe the deformation of the rod by a mapping $y_h:H^1(\Omega;\R^3)$, and the orientation of the LCE by a director field $n_h\in H^1(\Omega_0;\mathbb S^2)$. Here and below, $\mathbb S^2$ denotes unit-sphere in $\R^3$. We consider the energy functional 
\begin{equation}\label{eq:def:3d}
  \begin{aligned}
    \mathcal E_h(y_h,n_h):=&
    \frac{1}{h^2|S|}\int_{\Omega\setminus\Omega_0}W\Big(\bar x,\nabla_hy_h(x)\Big)\,dx+\frac{1}{h^2|S|}\int_{\Omega_0}W\Big(\bar x,L_h(n_h(x))^{-\frac12}\nabla_hy_h(x)\Big)\,dx\\
    &+\frac{\frankconstant{}^2}{|S_0|}\int_{\Omega_0}\big|\nabla_hn_h(x)(\nabla_hy_h(x))^{-1}\big|^2\det(\nabla_hy_h(x))\,dx.
  \end{aligned}
\end{equation}
The first integral is the elastic energy stored in the deformed material with reference configuration $\Omega\setminus\Omega_0$. 
Here, $\nabla_h:=(\partial_1,\frac1h\bar\nabla)$, $\bar\nabla:=(\partial_2,\partial_3)$ denotes the scaled gradient which emerges when passing from $\Omega_h$ to $\Omega$.
The second integral is the elastic energy of the nematic LCE. It invokes the step-length tensor 
\begin{equation*}
  L_h:\mathbb S^2\to\R^{3\times 3}_{\sym},\qquad L_h(n):=(1+h\bar r)^{-\frac13}(I + h\bar r n\otimes n),
\end{equation*}
which has been introduced by \cite{bladon1993transitions} to model the elastic-nematic coupling. The last integral is the (one-constant approximation of the) Frank-Oseen energy pulled back to the reference configuration.
As already mentioned, $\bar r\in\R$ and $\frankconstant{}>0$ denote model parameters.

We assume that the stored energy function $W:S\times\R^{3\times 3}\to[0,\infty]$ is Borel-measurable and satisfies for some $q>4$ and $C>0$, and for all $F\in\R^{3\times 3}$ and a.e.~$\bar x\in S$,
\begin{align}\tag{W1}
  &W(\bar x,RF)=W(\bar x,F)\qquad\text{for all }R\in\SO 3,\\\tag{W2}
  &W(\bar x,F)\geq \frac{1}{C} \dist^2(F,\SO 3)\text{ and } W(\bar x, I)=0,\\\tag{W3}
  &W(\bar x,\cdot)\text{ is $C^2$ in $\{F\in\R^{3\times 3}\,:\,\dist(F,\SO 3)<\tfrac1C\,\}$,}\\
  \tag{W4}
  &W(\bar x,F)\geq
    \begin{cases}
      \frac{1}{C}\max\{|F|^{q},\det(F)^{-\frac{q}{2}}\}-C &\text{if }\det F>0,\\
      +\infty&\text{else.}
    \end{cases}
\end{align}
The stored energy function thus describes a frame indifferent, cf.~(W1), material with a stress-free, non-degenerate reference configuration, cf.~(W2). Furthermore, by (W3) the material law is linearizable at identity and the linearization is Korn-elliptic, i.e., the quadratic form defined by
\begin{equation}\label{eq:def:Q}
  Q:S\times\R^{3\times 3}\to\R,\qquad Q(\bar x,G):=\frac12\nabla^2W(\bar x,I)G\cdot G,
\end{equation}
vanishes on $\R^{3\times 3}_{\Skew}$ and is positive definite on $\R^{3\times 3}_{\sym}$.
We remarkt that \eqref{eq:def:3d} is precisely the analogue for rods of the energy functional considered in \cite{bartels2022nonlinear} where the a bending model for LCE-plates is studied. 
The next theorem shows that $\mathcal E_h$ $\Gamma$-converges as $h\to 0$ to the one-dimensional limiting model  \eqref{def:limitI}:
\begin{theorem}[Derivation via dimension reduction]\label{T1}\mbox{}
  Let the cross-section $S$ and $S_0\subseteq S$ satisfy \eqref{ass:S}, \eqref{ass:S0}, and let $W$ satisfy (W1) -- (W4). Let $\mathcal E:\mathcal A\to\R$ be defined by \eqref{def:limitI} with $\bar Q$, $K_{\pre}$ and $E_{\res}$ given by Definition~\ref{D:effective}. Then the following properties hold:
  \begin{enumerate}[(a)]
  \item (Compactness). Suppose $\limsup\limits_{h\to 0}\mathcal E_h(y_h,n_h)<\infty$. Then there exists $(y,R,n)\in\mathcal A$ such that for a subsequence (not relabeled), we have
    \begin{align}\label{eq:comp1}
      (y_h-\fint_\Omega y_h,\nabla_hy_h)\to &(y,R)\qquad\text{in }L^2,\\
      \label{eq:comp2}
      n_h\to &n\qquad\text{in }L^2.
    \end{align}
  \item (Lower bound). Let $(y_h,n_h)\subset H^1(\Omega;\R^3)\times H^1(\Omega_0;\mathbb S^2)$ and $(y,R,n)\in\mathcal A$. Suppose that $(y_h,\nabla_hy_h,n_h)\to(y,R,n)$ strongly in $L^2$. Then
    \begin{equation*}
      \liminf\limits_{h\to 0}\mathcal E_h(y_h,n_h)\geq \mathcal E(y,R,n).
    \end{equation*}
  \item (Upper bound). Let $(y,R,n)\in\mathcal A$. Then there exists a sequence $(y_n,n_h)\subset H^1(\Omega;\R^3)\times H^1(\Omega_0;\mathbb S^2)$ such that $(y_h,\nabla_hy_h,n_h)\to(y,R,n)$ strongly in $L^2$ and
    \begin{equation*}
      \lim\limits_{h\to 0}\mathcal E_h(y_h,n_h)=\mathcal E(y,R,n).
    \end{equation*}
  \end{enumerate}
\end{theorem}
For the proof see Section~\ref{S:proof}. 
\medskip

We can take also clamped boundary conditions for the deformation into account:
\begin{proposition}[Clamped boundary conditions]\label{P1}
  Let $y_{bc}\in H^2(\omega;\R^3)$, $R_{bc}\in H^1(\omega;\SO 3)$, and suppose that $\partial_1y_{bc}=R_{bc}e_1$.
  \begin{enumerate}[(a)]
  \item Consider the situation of Theorem~\ref{T1} (a) and additionally suppose that
    \begin{equation}\label{eq:bc3d}
      y_h(0,\bar x)=y_{bc}(0)+hR_{bc}(0)\barx\qquad\text{a.e.~in~}S,
    \end{equation}
    where here and below, we write $\barx$ for the map $S\ni\bar x\mapsto (0,\bar x)^\top\in\R^3$.
    Then there exists $(y,R,n)\in\mathcal A$ with
    \begin{equation}\label{eq:bc1d}
      (y(0),R(0))=(y_{bc}(0),R_{bc}(0)),
    \end{equation}
    such that for a subsequence (not relabeled), we have
    \begin{align}\label{eq:compbc}
      (y_h,\nabla_hy_h,n_h)\to &(y,R,n)\qquad\text{in }L^2.
    \end{align}
  \item Consider the situation of Theorem~\ref{T1}  (c) and suppose that $(y,R,n)\in\mathcal A$ satisfies \eqref{eq:bc1d}. Then there exists a sequence $(y_n,n_h)\subset H^1(\Omega;\R^3)\times H^1(\Omega_0;\mathbb S^2)$ that satisfies \eqref{eq:bc3d} such that $(y_h,\nabla_hy_h,n_h)\to(y,R,n)$ strongly in $L^2$ and
    \begin{equation*}
      \lim\limits_{h\to 0}\mathcal E_h(y_h,n_h)=\mathcal E(y,R,n).
    \end{equation*}
  \end{enumerate}
\end{proposition}
For the proof see Section~\ref{S:proof}. 
\medskip
Next, we discuss soft anchoring conditions for the director $n_h$.
They come in the form of an additional contribution to the energy functional that penalizes deviations of the director $n_h$ from a prescribed configuration $\hat n_{bc}$.
In this context, it is natural to describe the director in local coordinates. 
In the following we first introduce  a general form of a soft anchoring condition for the three-dimensional model: Let $|\cdot|_a$ denote a semi-norm on $\R^3$, and let $\rho\in L^1(\Omega_0)$ denote a non-negative weight and set $\bar\rho(x_1):=\int_{S_0}\rho(x_1,\bar x)\,d\bar x$. For $(y_h,n_h)$ with $\mathcal E_h(y_h,n_h)<\infty$ and $\hat n_{bc}\in H^1(\omega;\mathbb S^2)$ define
\begin{equation*}
  \mathcal G_h(y_h,n_h;\hat n_{bc}):=\int_{\Omega_0}\left|\frac{\nabla_hy_h^\top n_h}{|\nabla_hy_h^\top n_h|}-\hat n_{bc}  \right|_a^2\rho\,dx.
\end{equation*}
Furthermore, for $(y,R,n)\in\mathcal A$ and $\hat n_{bc}\in H^1(\omega;\mathbb S^2)$ define
\begin{align*}
  \mathcal G_{\rm weak}(y,R,n;\hat n_{bc}):=&\int_\omega\left|R^\top n-\hat n_{bc}  \right|_a^2\bar\rho\,dx_1,\\
  \mathcal G_{\rm strong}(y,R,n;\hat n_{bc}):=&
                                                \begin{cases}
                                                  0&\text{if }\left|R^\top n-\hat n_{bc}  \right|_a^2\bar\rho=0\text{ a.e. in }\omega,\\
                                                  \infty&\text{else.}
                                                \end{cases}
\end{align*}

\begin{proposition}[Derivation of anchorings]\label{P2}
  \begin{enumerate}[(a)]
  \item The statements of Theorem~\ref{T1} hold with $(\mathcal E_h,\mathcal E)$ being replaced by $(\mathcal E_h + \mathcal G_h , \mathcal E + \mathcal G_{\rm weak} )$.
  \item Let $0<\beta<1$. The statements of Theorem~\ref{T1} hold with $(\mathcal E_h,\mathcal E)$ being replaced by $(\mathcal E_h + h^{-\beta}\mathcal G_h , \mathcal E + \mathcal G_{\rm strong} )$.
  \end{enumerate}
\end{proposition}
For the proof see Section~\ref{S:proof}. 
\medskip

\begin{example}
  \begin{enumerate}[(a)]
  \item (Full anchoring). In the case  $|\cdot|_a:=|\cdot|$ and $\rho\equiv |S_0|^{-1}$ we obtain
    \begin{align*}
      \mathcal G_{\rm weak}(y,R,n;\hat n_{bc})=\int_\omega\left|R^\top n-\hat n_{bc}  \right|^2\,dx_1,
    \end{align*}
    which penalizes deviations of the director from $\hat n_{bc}$. The corresponding strong anchoring enforces the director (in local coordinates) to be equal to $\hat n_{bc}$ a.e.~in $\omega$.
  \item (Tangentiallity). Consider  $|\xi|_a^2:=(\xi\cdot e_2)^2+(\xi\cdot e_3)^2$, $\rho\equiv |S_0|^{-1}$ and $\hat n_{bc}:=e_1$. We obtain
    \begin{align*}
      \mathcal G_{\rm weak}(y,R,n;\hat n_{bc})=\int_\omega(n\cdot Re_2)^2+(n\cdot Re_3)^2\,dx_1,
    \end{align*}
    which penalizes the non-tangential components of the director. The corresponding strong anchoring enforces the director (in local coordinates) to be tangential, i.e., $n=y'$.
  \item (Normality). Consider  $|\xi|_a^2:=(\xi\cdot e_1)^2$, $\rho\equiv |S_0|^{-1}$ and $\hat n_{bc}:=e_2$. We obtain
    \begin{align*}
      \mathcal G_{\rm weak}(y,R,n;\hat n_{bc})=\int_\omega(n\cdot y')^2\,dx_1,
    \end{align*}
    which penalizes the tangential components of the director. The corresponding strong anchoring enforces the director (in local coordinates) to be normal to the tangent.
  \end{enumerate}
\end{example}

\subsection{Definition and evaluation of the effective coefficients}\label{S:4}
In the following we present the definition of the effective coefficients $\bar Q$, $K_{\rm eff}$ and $E_{\res}$. We first define the effective coefficients abstractly based on a projection scheme introduced in \cite{bauer2019derivation}. We then characterize the coefficients with help of cell-problems and correctors. Finally, we derive more explicit formulas in the special case of an isotropic material. Throughout this section we assume that $W$ satisfies (W1) -- (W4) and that $Q$ is defined by \eqref{eq:def:Q}.
\smallskip

For the abstract definition let $\mathbb H:=L^2(S;\R^{3\times 3}_{\sym})$ denote the Hilbert space with scalar product
\begin{equation*}
  \big(\Psi,\Phi\big)_Q:=\frac12\fint_S\mathbb L(\bar x)\Psi(\bar x)\cdot \Phi(\bar x)\,d\bar x,
\end{equation*}
where $\frac12\mathbb L(\bar x)G\cdot G'=\frac12\nabla^2 W(\bar x,I)G\cdot G'$ for all $G,G'\in\R^{3\times 3}$. Note that the associated norm is given by $\|\Psi\|_Q^2=\fint_SQ(\bar x,\Psi(\bar x))\,d\bar x$. 
We consider the subspaces
\begin{align*}
  \Hrel:=&\Big\{S\ni\bar x\mapsto ae_1\otimes e_1+\sym(0,\bar\nabla\varphi(\bar x))\,:\,a\in\R,\,\varphi\in H^1(S;\R^3)\,\Big\},\\
  \Hmicro:=&\Big\{S\ni\bar x\mapsto\sym((K\barx)\otimes e_1)+\chi(\bar x)\,:\,K\in\R^{3\times 3}_{\Skew},\,\chi\in{\Hrel}\,\Big\}.
\end{align*}
With help of Korn's inequality we deduce the following statement (whose elementary proof we leave to the reader):
\begin{lemma}\label{L:Hrelrep}
  Let $H^1_{\rm av}(S;\R^3)$ denote the space of functions $\varphi\in H^1(S;\R^3)$ satisfying
  \begin{equation*}
    \label{eq:zeroav2}
    \fint_S\varphi=0\text{ and }\fint_S\partial_3\varphi_2-\partial_2\varphi_3\,d\bar x=0.
  \end{equation*}
  Then $H^1_{\rm av}(S;\R^3)$ equipped with the norm $\varphi\mapsto \left(\fint_S|\sym(0,\bar\nabla\varphi)|^2\,d\bar x\right)^\frac12$ is a Hilbert space. Furthermore, the map
  \begin{equation*}
    \R\times H^1_{\rm av}(S;\R^3)\ni(a,\varphi)\mapsto \sym\big(ae_1,\bar\nabla\varphi\big)\in\Hrel
  \end{equation*}
  is an isomorphism.
\end{lemma}
The previous lemma implies that $\Hrel$ and $\Hmicro$ are closed subspaces of $\mathbb H$. Thus, the subspaces $\mathbb H_{\rm res}$ and $\Hmacro$ defined by the $(\cdot,\cdot)_Q$-orthogonal decompositions
\begin{equation*}
  \mathbb H=\Hmicro\oplus\mathbb H_{\rm res},\qquad \Hmicro=\Hmacro\oplus{\Hrel},
\end{equation*}
are closed as well.
In the following, we write $P_X$ for the orthogonal projection onto a closed subspace $X\subset\mathbb H$. 

We recall the following result, which is a special case of \cite[Lemma~2.10]{bauer2019derivation}:
\begin{lemma}
  The map
  \begin{equation*}
    \mathbf E:\R^{3\times 3}_{\Skew}\to   \Hmacro,\qquad \mathbf E(K):=P_{\Hmacro}\Big(\sym\big((K\barx)\otimes e_1\big)\Big)
  \end{equation*}
  is a linear isomorphism. 
\end{lemma}
We are now in position to define the effective coefficients as follows:
\begin{definition}[Effective coefficients]\label{D:effective}
  \begin{align*}
    &\bar Q:\R^{3\times 3}_{\Skew}\to\R,\qquad \bar Q(K):=\|\mathbf E(K)\|_{Q}^2,\\
    &K_{\pre}:\R^{3\times 3}_{\dev}\to\R^{3\times 3}_{\Skew},\qquad      K_{\pre}(U):=(\mathbf E^{-1}\circ P_{\Hmacro})\Big(\tfrac12  \boldsymbol{1}_{S_0} U\Big),\\
    &E_{\res}:\R^{3\times 3}_{\dev}\to\R,\qquad      E_{\res}(U):=\|P_{\Hres}\Big(\tfrac12  \boldsymbol{1}_{S_0} U\Big)\|^2_Q.
  \end{align*}
\end{definition}
The definition is motivated by the following relaxation result:
\begin{lemma}[Relaxation formula]\label{L:relax}
  Let $K\in\R^{3\times 3}_{\Skew}$ and $U\in \R^{3\times 3}_{\dev}$. Then
  \begin{align*}
    &\inf_{\chi\in\Hrel}\fint_S Q\Big(\bar x,\sym\big((K\barx)\otimes e_1\big)+\tfrac{\bar r}{2}  \boldsymbol{1}_{S_0} U+\chi\Big)\\
    &\qquad =\bar Q(K+\bar r K_{\pre}(U))+\bar r^2E_{\res}(U).
  \end{align*}
\end{lemma}
Thanks to the assumptions on $W$ and $S$ we obtain the following properties:
\begin{lemma}
  The maps $\bar Q$ and $E_{\res}$ are quadratic, and $K_{\pre}$ is  linear. Moreover, there exists $C>0$ only depending on $W$ and $S$ such that
  \begin{alignat*}{2}
    \frac{1}{C}|K|^2&\leq \bar Q(K)\leq C|K|^2,\\
    0&\leq E_{\res}(U)\leq C|U|^2,\\
    |K_{\pre}(U)|&\leq C|U|.
  \end{alignat*}
\end{lemma}
We omit the proof, since it is similar to \cite[Lemma~2.6]{bartels2022nonlinear}.
\medskip

Next, we introduce a scheme to evaluate these quantities. The scheme invokes $3+5$ correctors that only depend on $Q$, $S$ and $S_0$, and are defined with help of linear elliptic systems on the domain $S$. We start by representing the effective coefficients in coordinates. To that end, we consider  the following orthonormal basis of $\R^{3\times 3}_{\dev}$,
\begin{gather*}
  U_1:=\sqrt{2}\sym(e_3\otimes e_2),\quad U_2:=\sqrt{2}\sym(e_2\otimes e_1),\quad U_3:=\sqrt{2}\sym(e_3\otimes e_1),\\
  U_4:=\sqrt{\frac{2}{3}}(e_1\otimes e_1-\frac12e_2\otimes e_2-\frac12e_3\otimes e_3),\quad U_5:=\frac{1}{\sqrt 2}(e_2\otimes e_2-e_3\otimes e_3),
\end{gather*}
and the following orthonormal basis of $\R^{3\times 3}_{\Skew}$,
\begin{equation*}
  K_1=\frac{1}{\sqrt 2}\big(e_2\otimes e_3-e_3\otimes e_2\big),\quad K_2=\frac{1}{\sqrt 2}\big(e_1\otimes e_2-e_2\otimes e_1\big),\quad K_3=\frac{1}{\sqrt 2}\big(e_1\otimes e_3-e_3\otimes e_1\big).
\end{equation*}

\todoneukamm{Note:
  \protect\begin{equation*}
    \sym(K_1\barx\otimes e_1)=
    \frac{1}{2\sqrt 2}\protect\begin{pmatrix}
      0&x_3&-x_2\\
      x_3&0&0\\
      -x_2&0&0
      \protect\end{pmatrix}=\frac{1}{\sqrt 2}\big(x_3\sym(e_1\otimes e_2)-x_2\sym(e_1\otimes e_3)\big)
    \protect\end{equation*}
  \protect\begin{equation*}
    \sym(K_2\barx\otimes e_1)=
    \frac{x_2}{\sqrt 2}e_1\otimes e_1,\qquad     \sym(K_3\barx\otimes e_1)=
    \frac{x_3}{\sqrt 2}e_1\otimes e_1
    \protect\end{equation*}
}
\begin{lemma}[Coordinatewise representation]\label{L:coeffrep}\mbox{}
  \begin{enumerate}[(a)]
  \item For $i=1,2,3$ consider
    \begin{equation*}
      \Psi_i:=P_{\Hmacro}\Big(\sym\big(K_i\barx\otimes e_1\big)\Big)
    \end{equation*}
    and define $\mathbb Q\in\R^{3\times 3}_{\sym}$ and $\mathbb U\in\R^{3\times 5}$ as
    \begin{eqnarray*}
      \mathbb Q_{ij}&:=&\big(\Psi_i,\Psi_j\big)_{Q},\qquad\text{for }i,j=1,2,3,\\
      \mathbb U_{ij}&:=&\big(\boldsymbol{1}_{S_0}U_j,\Psi_i\big)_{Q}\qquad\text{for }i=1,2,3,\,j=1,\ldots,5.
    \end{eqnarray*}
    Then for all $K\in\R^{3\times 3}_{\Skew}$ and $U\in\R^{3\times 3}_{\dev}$ we have
    \begin{eqnarray}\label{eq:formQ}
      \bar Q(K)&=&k\cdot\mathbb Qk,\qquad k:=(K\cdot K_1,\ldots,K\cdot K_3)^\top,\\
      \label{eq:Kpre}
      K_{\pre}(U)&=&\sum_{i=1}^3(\mathbb Q^{-1}\mathbb Uu)_iK_i,\qquad  u:=\frac12\big(U\cdot U_1,\ldots,U\cdot U_5\big)^\top.
    \end{eqnarray}
  \item For $j=1,\ldots,5$ consider
    \begin{equation*}
      \Phi_j:=\boldsymbol 1_{S_0}U_j-P_{\Hrel}\big(\boldsymbol 1_{S_0}U_j\big)-\sum_{i=1}^3(\mathbb Q^{-1}\mathbb U)_{ij}\Psi_i,
    \end{equation*}
    and define $\mathbb E_{\res}\in\R^{5\times 5}_{\sym}$ as
    \begin{equation*}
      \mathbb E_{\res,ij}:=\big(\Phi_i,\Phi_j\big)_Q,\qquad\text{for }i,j=1,\ldots,5.
    \end{equation*}
    Then for all $U\in\R^{3\times 3}_{\dev}$ we have
    \begin{equation}\label{eq:Eres}
      E_{\res}(U)=u\cdot\mathbb E_{\res}u,,\qquad  u:=\frac12\big(U\cdot U_1,\ldots,U\cdot U_5\big)^\top.
    \end{equation}
  \end{enumerate}
\end{lemma}
The orthogonal projections onto $\Hmacro$ and $\Hrel$ appearing in the definition of $\Psi_i$ and $\Phi_j$ lead to corrector problems that take the form of quadratic minimization problems whose solutions are characterized by linear elliptic systems:
\begin{lemma}[Corrector equations]\label{L:correctors1}
  For $i=1,2,3$ and $j=1,\ldots,5$, let $(a_{i},\varphi_i)$ and $(a_{U_j},\varphi_{U_j})$ denote the unique minimizer in $\R\times H^1_{\rm av}(S;\R^3)$ of the functional
  \begin{align}\label{eq:corr1}
    &(a,\varphi)\mapsto \fint_SQ\Big(\bar x,F+(ae_1,\bar\nabla\varphi)\Big)\,d\bar x,
  \end{align}
  with $F=\sym(K_i\barx\otimes e_1)$ and $F=\boldsymbol{1}_{S_0}U_j$, respectively.
  Then
  \begin{eqnarray*}
    \Psi_i&=&\sym\Big(K_i\barx\otimes e_1+(a_ie_1,\bar\nabla\varphi_i)\Big),\\
    \Phi_j&=&\boldsymbol 1_{S_0}U_j+\sym\big(a_{U_j},\bar\nabla\varphi_{U_j}\big)-\sum_{i=1}^3\big(\mathbb Q^{-1}\mathbb U\big)_{ij}\Psi_i.
  \end{eqnarray*}
\end{lemma}
% 
% \begin{remark}
%   The minimization problem \eqref{eq:corr1} admits a unique solution $(a,\varphi)\in \R\times H^1_{\rm av}(S;\R^3)$. It is characterized by the Euler-Lagrange equation
%   \begin{equation*}
%     \int_S\mathbb L(\bar x)\big(F+(ae_1,\bar\nabla\varphi)\big)\cdot(\tilde ae_1,\bar\nabla\tilde\varphi)\,d\bar x=0\qquad\text{for all }\tilde a\in\R,\,\tilde\varphi\in H^1_{\rm av}(S;\R^3).
%   \end{equation*}
%   We note that it suffices to consider test functions $\tilde\varphi\in C^1(\bar S;\R^3)$.
% \end{remark}
% 
\subsection{The special case of an isotropic material with circular cross-section}
We consider the special case of a homogeneous, isotropic material, i.e,
\begin{equation}\label{eq:iso}
  Q(G)=\frac{\lambda}{2}(\trace G)^2+\mu|\sym G|^2.
\end{equation}
In that case the formulas for  $\bar Q$ and $K_{\pre}$ become more explicit. They further simplify if we consider bi-rods with %a special geometry of the cross-section, say % . To illustrate this we consider the case of % We consider the following two settings: The bilayer case,
% % \begin{equation}\label{eq:bilayer}
% %   S=(-\frac\ell2,\frac\ell2)\times(-\tfrac12,\tfrac12),\qquad S_0=(-\frac\ell 2,\frac\ell2)\times(0,\tfrac12),
% % \end{equation}
% % or the case of
a circular cross-section:
% \begin{equation}\label{eq:circular}
%   S=B(0;1),\qquad S_0:=B(0;1)\cap\{\bar x=(x_2,x_3)\,:\,x_2>0\}.
% \end{equation}
% 

\begin{lemma}[The isotropic case and the case with a circular cross-section]\label{L:isotropic}
  Let $\alpha_S\in H^1(S)$ denote the unique minimizer to
  \begin{equation}\label{eq:alpha}
    \fint_S\Big|
    \begin{pmatrix}
      \partial_2\alpha_S\\\partial_3\alpha_S
    \end{pmatrix}
    +\frac{1}{\sqrt 2}
    \begin{pmatrix}
      x_3\\-x_2
    \end{pmatrix}\Big|^2\,d\bar x\qquad\text{subject to }\int_S\alpha_S\,d\bar x=\int_S\partial_2\alpha_S\,d\bar x=\int_S\partial_3\alpha_S\,d\bar x=0,
  \end{equation}
  and set
  \begin{equation*}
    c_S:=      \fint_S\Big|
    \begin{pmatrix}
      \partial_2\alpha_S\\\partial_3\alpha_S
    \end{pmatrix}
    +\frac{1}{\sqrt 2}
    \begin{pmatrix}
      x_3\\-x_2
    \end{pmatrix}\Big|^2\,d\bar x.
  \end{equation*}
  Assume \eqref{eq:iso}. Then:
  \begin{enumerate}[(a)]
  \item We have $\mathbb Q=\operatorname{diag}(q_1,q_2,q_3)$ where
    \begin{align*}
      &q_1:=\frac{\mu}{2} c_S,\qquad q_2:=\frac{\mu(3\lambda+2\mu)}{\lambda+\mu}\fint_S\frac{x_2^2}{4}\,d\bar x,\qquad q_3:=\frac{\mu(3\lambda+2\mu)}{\lambda+\mu}\fint_S\frac{x_3^2}{4}\,d\bar x,
    \end{align*}
    and
    \begin{equation*}
      \begin{aligned}
        K_{\pre}(U)=\,&\,
        \Big(u_2\big(\tfrac{\sqrt 2}{|S|c_S}\int_{S_0}\partial_2\alpha_S{+}\tfrac{x_3}{\sqrt 2}\,d\bar x\big)+u_3
        \big(\tfrac{\sqrt 2}{|S|c_S}\int_{S_0}\partial_3\alpha_S{-}\tfrac{x_2}{\sqrt 2}\,d\bar x\big)\Big)K_1\\
        \,&\,+u_4\Big(\frac{2}{\sqrt 3}\frac{\int_{S_0}x_2\,d\bar x}{\int_Sx_2^2\,d\bar x}\Big)K_2
        +u_4\Big(\frac{2}{\sqrt 3}\frac{\int_{S_0}x_3\,d\bar x}{\int_Sx_3^2\,d\bar x}\Big)K_3,
      \end{aligned}
    \end{equation*}
    where $u_1,\ldots,u_5$ are defined as in \eqref{eq:Kpre}.
    % \item In the bilayer case, i.e., when $  \boldsymbol{1}_{S_0}$ is independent of $x_2$, we have $K_{\pre}(U)\cdot K_2=0$.
  \item In the case of a circular cross-section, $S=B(0;\pi^{-\frac12})$,  we have $\alpha_S=0$, $c_S=\frac{1}{4\pi}$, and
    % \todoneukamm{Note the corrected value of $c_S$ (due to the additional factor of $\frac1{\sqrt 2}$ in the definition of $c_S$)}
    \begin{equation*}
      q_1=\frac{1}{8\pi}\mu,\qquad q_2=q_3=\frac{1}{16\pi}\frac{\mu(3\lambda+2\mu)}{\lambda+\mu}.
    \end{equation*}
    % \todoneukamm{Note that $\frac{\mu(3\lambda+2\mu)}{\lambda+\mu}=\mu(2+\frac{\lambda}{\lambda+\mu})>2\mu$, and thus $q_2=q_3>q_1$.}
    If in addition $S_0=S\cap\{x_3\geq 0\}$, then
    % \todoneukamm{Nachgerechnet.}
    \begin{equation*}
      K_{\pre}(U)=\,\,
      \frac{8 u_2}{3\sqrt\pi}K_1+\frac{16 u_4}{3\sqrt{3\pi}}K_3,
    \end{equation*}
    where $u_1,\ldots,u_5$ are defined as in \eqref{eq:Kpre}.
  \end{enumerate}
\end{lemma}

\section{Simulation and model exploration}\label{S:num}
For our numerical experiments, we use a discrete gradient flow approach based on the work in \cite{BarRei20} in order to numerically approximate critical points of the energy functional $\mathcal E$.
For convenience we use the notation $y', y'',\cdots$ to denote derivatives with regard to~$x_1$. Furthermore, in this section we use $h$ to denote the discretization scale (and not the thickness of the three-dimensional domain as in the previous section).

We first bring the energy functional $\mathcal E$ into a form that is similar to the one considered in \cite{BarRei20}. Note that for a framed curve $(y,R)$ the columns of the rotational frame take the form
\begin{equation*}
  R=\big(y',b,y'\wedge b),\qquad\text{where }b:=Re_2.
\end{equation*}
We may introduce two bending components and a twist rate of the curve via
\begin{equation*}
  \kappa_b:=y''\cdot b,\qquad \kappa_d:=y''\cdot (y'\wedge b)\qquad\text{and}\qquad\beta:=b'\cdot (y'\wedge b),
\end{equation*}
and deduce that
\begin{equation*}
  K=R^\top\partial_1R=
  \begin{pmatrix}
    0&-\kappa_b&-\kappa_d\\
    \kappa_b&0&\beta\\
    \kappa_d&-\beta&0
  \end{pmatrix}=\sqrt{2}\big(\beta K_1+\kappa_b K_2+\kappa_d K_3\big),
\end{equation*}
where $K_1,K_2,K_3$ denote the orthonormal basis of $\R^{3\times 3}_{\Skew}$ introduced above.

Motivated by this we introduce the functional
% \todoneukamm{modulo a factor of $2$}
\begin{align*}
  \bar{\mathcal E}(y,b,\hat n)= \ &\frac{1}{2}\int_{\omega} \bar{Q}(K + \bar{r}K_{\operatorname{pre}}(U(\hat{n}))) + \bar{r}^2 E_{\operatorname{res}}(U(\hat{n})) \,dx_1\\
                                  &+\frac{1}{2}\frankconstant{}^2 \int_{\omega}|(R\hat n)'|^2\,dx_1,
\end{align*}
% \todoneukamm{Check prefactor $\frac{|S_0|}{|S|}$.}
where
\begin{equation}\label{eq:relation}
  U(\hat{n}) := \frac{1}{3} \id - \hat{n} \otimes \hat{n},\qquad \hat{n} := R^\top n,\qquad R := (y',b,y'\wedge b),
\end{equation}
and note that we have $\mathcal E(y,R,n)=2\bar{\mathcal E}(y,b,\hat n)$ provided $y,R,b,n$ and $\hat n$ are related by \eqref{eq:relation}.

We note that with help of $\hat n$ (which is just the LCE-director $n$ expressed in local coordinates), the terms $ K_{\operatorname{pre}}(U(\hat{n})) $ and $ E_{\operatorname{res}}(U(\hat{n})) $ become independent of $y$ and $R$---a property that will simplify the form of the gradient of $\bar{\mathcal E}$.

In the isotropic case, which we shall consider from now on, the expression further simplifies by appealing to Lemma~\ref{L:isotropic}, and we obtain
\begin{align*}
  \bar{\mathcal E}(y,b,\hat n)= \ &\frac{1}{2}\int_{\omega} \bar{q}_1|\beta-\bar r k_1(U(\hat{n}))|^2 + \bar{q}_2|\kappa_b-\bar r k_2(U(\hat{n}))|^2 \\
                                  &\ \ +\bar{q}_3|\kappa_d-\bar r k_3(U(\hat{n}))|^2 + \bar{r}^2 E_{\operatorname{res}}(U(\hat{n})) \,dx_1\\
                                  &+\frac{1}{2}\frankconstant{}^2 \int_{\omega}|(R\hat n)'|^2\,dx_1,
\end{align*}
where for $ i = 1,2,3 $, the linear maps $ k_i : \R^{3\times 3}_{\operatorname{dev}} \rightarrow \R $ are given by $ k_i(U) = \frac{1}{\sqrt{2}}K_{\operatorname{pre}}(U) \cdot K_i $.
By appealing to binomial formulas and the relations $ |y''|^2 = |\kappa_b|^2 + |\kappa_d|^2 $ and $ |b'|^2 = |\kappa_b|^2 + |\beta|^2 $, we eventually get
\begin{equation}\label{eq:num:barE}
  \begin{aligned}
    \bar{\mathcal E}(y,b,\hat n)=\ & \frac12\int_{\omega}\bar{q}_3 |y''|^2+\bar{q}_1 |b'|^2 + \frankconstant{}^2 |(R\hat n)'|^2\,dx_1\\
    &+G(y,b) + \bar{\mathcal E}_{\operatorname{res}}(\hat{n}) + N(y,b,\hat{n})
  \end{aligned}
\end{equation}
with the functionals
\begin{align*}
  G(y,b) = \ & \frac12 (\bar{q}_2-\bar{q}_1-\bar{q}_3)\int_{\omega} |y''\cdot b|^2\,dx_1,\\
  \bar{\mathcal E}_{\operatorname{res}}(\hat{n}) = \ & \frac{1}{2}\bar{r}^2 \int_{\omega} E_{\operatorname{res}}(U(\hat{n})) \,dx_1\,\,\,\,\,\, \textrm{and}\\
  N(y,b,\hat{n}) = \ & \int_\omega \frac{1}{2} \bar{r}^2\sum_{i=1}^{3} \bar{q}_i|k_i(U(\hat{n}))|^2 - \bar{r} \bar{q}_1 (b'\cdot (y' \wedge b)) k_1(U(\hat{n}))\\
             &\ \ \ \ - \bar{r} \bar{q}_2 (y''\cdot b) k_2(U(\hat{n})) - \bar{r} \bar{q}_3 (y''\cdot (y' \wedge b)) k_3(U(\hat{n}))\,dx_1.
\end{align*}
The structure of the energy functional \eqref{eq:num:barE} is similar to the bending-twisting energy that was used in \cite{BarRei20} with the difference that we now have additional terms which depend on the LCE-director $ \hat{n} $.
                                                 %                                                  Except for the constants, the first line of \eqref{eq:num:barE} now contains the same terms as the energy used in \cite{BarRei20}. The other terms depend on the director $ \hat{n} $ and are therefore new. 
                                                 %                                                  This energy is quite similar to what was used in [include Reference to 'NUMERICAL SOLUTION OF A BENDING-TORSION MODEL FOR ELASTIC RODS' by Bartels and Reiter 2020]
                                                 %                                                  with the difference that
                                                 %                                                  \todoneukamm{add some words about the difference}.

                                                 %                                                  \todoneukamm{$(w,\mu,m)$ ist eine sehr schlechte Wahl fuer die Notation, denn $w$ sieht aus wie $\omega$, $\mu$ ist Lam\'e-Constante. Kanonisch w\"are: $(\delta y,\delta b,\delta\hat n)$ fuer die Inkremente.
                                                 %                                                  Ich wuerde auch folgende Notation verwenden:
                                                 %                                                  \protect\begin{equation*}
                                                 %                                                  \frac{\partial\bar{\mathcal E}(y,b,\hat n)}{\partial y}[\delta y]\qquad\text{or}\qquad \frac{d}{ds}\bar{\mathcal E}(y+s\delta y,b,\hat n)\vert_{s=0}
                                                 %                                                  \protect\end{equation*}
                                                 %                                                  Am besten mit Soeren besprechen.\\
    %     } 

\subsection{Numerical minimization by a discrete gradient flow}
Next, we introduce a suitable discretisation. For the approximation of the deformation $ y $, our approach uses piecewise cubic, $ C^1 $-conforming elements, whereas the frame director $ b $ and the LCE-director $ \hat{n} $ are approximated via piecewise linear, continuous elements. %the P1-FEM and linear splines.
More specifically, following \cite{bartels2020finite} we consider a partitioning of $\bar\omega = [0,L]$ defined by sets of nodes $\mathcal N_h$ and elements $\mathcal T_h$, and denote by
\begin{align*}
  \mathcal S^{1,0}(\mathcal T_h)=&\{b_h\in C^0(\bar\omega)\}\,:\,b_h\vert_T\in P_1(T)\text{ for all }T\in\mathcal T_h\},\\
  \mathcal S^{3,1}(\mathcal T_h)=&\{y_h\in C^1(\bar\omega)\}\,:\,y_h\vert_T\in P_3(T)\text{ for all }T\in\mathcal T_h\},
\end{align*}
the associated spaces of piecewise linear and continuous (resp.\ piecewise cubic and $C^1$-conforming) finite elements. Moreover, we introduce the discrete space
\begin{align*}
    %     (y_h,b_h,\hat{n}_h) \in 
  V^h_{\operatorname{LCE}} = \mathcal{S}^{3,1}(\mathcal{T}_h)^3 \times \mathcal{S}^{1,0}(\mathcal{T}_h)^3 \times \mathcal{S}^{1,0}(\mathcal{T}_h)^3.
\end{align*}
On this vector space we define a discrete energy functional $ \bar{\mathcal E}_{h,\eps} $, which contains the same terms as $ \bar{\mathcal E} $---in some cases with appropriate quadrature---as well as a penalty term to approximately incorporate the contraint $ y' \cdot b = 0 $: For $ (y_h,b_h,\hat{n}_h) \in V^h_{\operatorname{LCE}} $ let $ R_h = (y_h',b_h,y_h'\wedge b_h) $ and define
\begin{equation}\label{eq:num:barEh}
  \begin{aligned}
    \bar{\mathcal E}_{h,\eps}(y_h,b_h,\hat n_h)=\ & \frac12\int_{\omega}\bar{q}_3 |y_h''|^2+\bar{q}_1 |b_h'|^2 + \frankconstant{}^2 |(R_h\hat n_h)'|^2\,dx_1\\
    & + P_{h,\eps}(y_h,b_h)+G_h(y_h,b_h) + \bar{\mathcal E}_{\operatorname{res},h}(\hat{n}_h) + N_h(y_h,b_h,\hat{n}_h),
  \end{aligned}
\end{equation}
where the aforementioned penalty term is defined as
\begin{align*}
  P_{h,\eps}(y_h,b_h) = \frac{1}{2\eps}\int_{\omega} \mathcal{I}_h^{1,0} [(y_h'\cdot b_h)^2]\, dx_1.
\end{align*}
Above, $ \mathcal{I}_h^{1,0} $ denotes the nodal interpolation operator associated with $\mathcal{S}^{1,0}(\mathcal{T}_h)$ and $\eps>0$ a parameter to adjust the penalization. The functionals $ G_h $, $ \bar{\mathcal E}_{\operatorname{res},h} $ and $ N_h $ are discrete versions of $ G $,  $ \bar{\mathcal E}_{\operatorname{res}} $ and $ N $ respectively, that contain nodal interpolation operators on the related discrete spaces to simplify the computation of the integrals.
\todoneukamm{Definitionen angeben!}

The energy \eqref{eq:num:barEh} is minimized in the discretized admissible set
\begin{align*}
  \mathcal A_h:=\{(y_h,b_h,\hat{n}_h) \in V^h_{\operatorname{LCE}}\,:\,L_{\operatorname{BC}}(y_h,b_h,\hat{n}_h) = \ell_{\operatorname{BC}} ,\ \ \ \\
  |y_h'(z)| = |b_h(z)| = |\hat{n}(z)| = 1 \text{ f.a. } z \in \mathcal N_h\},
\end{align*}
where $ \mathcal{N}_h $ denotes the set of vertices related to $ \mathcal T_h $.
The expression $ L_{\operatorname{BC}}(y_h,b_h,\hat{n}_h) = \ell_{\operatorname{BC}} $ implies that $ (y_h,b_h,\hat{n}_h) $ fulfil the boundary conditions specified by $ \ell_{\operatorname{BC}} $. Different conditions such as fixed, clamped, free and periodic are possible. The boundary conditions for the individual variables are denoted by $ L_{\operatorname{BC},y} $, $ L_{\operatorname{BC},b} $ and $ L_{\operatorname{BC},\hat{n}} $.
    %     We use the nodal interpolant $ \mathcal{I}_h^{1,0} $ of $ \mathcal{S}^{1,0}(\mathcal{T}_h) $ and add the term
    %     
    %     \begin{align*}
    %     \frac{1}{2\eps}\int_{\omega} \mathcal{I}_h^{1,0} [(y_h'\cdot b_h)^2]\, dx_1
    %   \end{align*}
    %     to the energy as a penalty term for the condition $ (y' \cdot b) = 0 $.\\

    %     Our time discretisation uses a uniform partitioning $ 0 = t_0 < t_1 <\dots < t_M = T $, where $ M \in \mathbb{N} $ is the number of timesteps and $ \tau = T/M $ is the timestep size.\\%For simplicity's sake, we use a constant timestep size $ \tau $, i.e. $ t_k = k\tau $ for all $ k = 0,\dots,K $.\\
    %     At the time $ t_k = \tau k, k\in\{1,\dots,M\} $, the approximate solution $ (y_h^k,b_h^k,\hat{n}_h^k) \in \mathcal A_h $ is computed using a semi-implicit scheme based on gradient descent.

We employ a discrete gradient flow scheme to approximate minimizers of the discretized energy and incorporate linearized versions of the unit-length and boundary conditions by restricting each step of the iteration to a corresponding tangent space of the admissible set. For $ (y_h,b_h,\hat{n}_h) \in \mathcal A_h $, these tangent spaces are given by
\begin{align*}
  \mathcal F_{h,y}(y_h) &= \{\delta y_h \in \mathcal{S}^{3,1}(\mathcal{T}_h)^3 \,:\, L_{\operatorname{BC},y}(\delta y_h) = 0,\ y_h'(z) \cdot \delta y_h'(z) = 0 \text{ f.a. } z \in \mathcal N_h \},\\
  \mathcal F_{h,b}(b_h) &= \{\delta b_h \in \mathcal{S}^{1,0}(\mathcal{T}_h)^3 \,:\, L_{\operatorname{BC},b}(\delta b_h) = 0,\ b_h(z) \cdot \delta b_h(z) = 0 \text{ f.a. } z \in \mathcal N_h \}
\end{align*}
as well as
\begin{align*}
  \ \mathcal F_{h,\hat{n}}(\hat{n}_h) = \{\delta \hat{n}_h \in \mathcal{S}^{1,0}(\mathcal{T}_h)^3 \,:\,  L_{\operatorname{BC},\hat{n}}(\delta \hat{n}_h) = 0,\ \hat{n}_h(z) \cdot \delta \hat{n}_h(z) = 0 \text{ f.a. } z \in \mathcal N_h \}.
\end{align*}
Note that the functions in these tangent spaces are required to satisfy homogeneous versions of the given boundary conditions.

The variations of the energy with respect to the different variables are approximated semi-implicitly, where the convex quadratic terms are mostly handled implicitly while we rely on an explicit treatment of the nonlinear and non-convex parts. We let $ (\cdot,\cdot)_Y $, $ (\cdot,\cdot)_X  $ and $ (\cdot,\cdot)_Z  $ denote bilinear forms on the spaces $ \mathcal{S}^{3,1}(\mathcal{T}_h)^3 $, $ \mathcal{S}^{1,0}(\mathcal{T}_h)^3 $ and  $\mathcal{S}^{1,0}(\mathcal{T}_h)^3 $ respectively and use the backwards difference quotient $ d_t $.

\begin{algorithm}
  Set initial values $ (y_h^0,b_h^0,\hat{n}_h^0) \in	\mathcal A_h $, a timestep size $ \tau $, a stopping criterion $ \eps_{\operatorname{stop}} $ and initialize $ k = 1 $.
  \begin{enumerate}[(1)]
  \item Compute $ y_h^k = y_h^{k-1} + \tau d_ty_h^k $ with $ d_ty_h^k \in \mathcal F_{h,y}(y_h^{k-1})$ such that 
    \begin{align*}
      \left(d_ty_h^k,\delta y_h\right)_Y +\ & \bar{q}_3\left([y_h^k] '', \delta y_h''\right)_{L^2(\omega)} + \frac{\partial P_{h,\eps}(y_h^k,b_h^{k-1})}{\partial y_h^k}[\delta y_h]\\% + \eps^{-1} \left([y_h^k]'\cdot b_h^{k-1} , \delta y_h' \cdot b_h^{k-1} \right)\\
      =\ &- \frankconstant{}^2 \left((\frac{\partial R_h^{k-1}}{\partial y_h^{k-1}}[\delta y_h])\hat{n}_h^{k-1})' , (R_h^{k-1} \hat{n}_h^{k-1})' \right)_{L^2(\omega)}\\
                                            & - \frac{\partial G_h(y_h^{k-1},b_h^{k-1})}{\partial y_h^{k-1}}[\delta y_h]  - \frac{\partial N_h(y_h^{k-1},b_h^{k-1},\hat{n}_h^{k-1})}{\partial y_h^{k-1}}[\delta y_h]
    \end{align*}
    for all $ \delta y_h \in \mathcal F_{h,y}(y_h^{k-1}) $.
  \item Compute $ b_h^k = b_h^{k-1} + \tau d_tb_h^k $ with $ d_tb_h^k \in \mathcal F_{h,b}(b_h^{k-1})$ such that 
    \begin{align*}
      \left(d_tb_h^k,\delta b_h\right)_X +\ & \bar{q}_1\left([b_h^k] ', \delta b_h'\right)_{L^2(\omega)} + \frac{\partial P_{h,\eps}(y_h^k,b_h^k)}{\partial b_h^k}[\delta b_h]\\% + \eps^{-1} \left([y_h^k]'\cdot b_h^{k-1} , \delta y_h' \cdot b_h^{k-1} \right)\\
      =\ &- \frankconstant{}^2 \left((\frac{\partial R_h^{k-1}}{\partial b_h^{k-1}}[\delta b_h])\hat{n}_h^{k-1})' , (R_h^{k-1} \hat{n}_h^{k-1})' \right)_{L^2(\omega)}\\
                                            & - \frac{\partial G_h(y_h^{k},b_h^{k-1})}{\partial b_h^{k-1}}[\delta b_h]  - \frac{\partial N_h(y_h^{k},b_h^{k-1},\hat{n}_h^{k-1})}{\partial b_h^{k-1}}[\delta b_h]
    \end{align*}
    for all $ \delta b_h \in \mathcal F_{h,b}(b_h^{k-1}) $.
  \item Compute $ \hat{n}_h^k = \hat{n}_h^{k-1} + \tau d_t\hat{n}_h^k $ with $ d_t\hat{n}_h^k \in \mathcal F_{h,\hat{n}}(\hat{n}_h^{k-1})$ such that 
    \begin{align*}
      \left(d_t\hat{n}_h^k,\delta \hat{n}_h\right)_Z +\ & \frankconstant{}^2\left((R_h^k\hat{n}_h^k)', (R_h^k\delta \hat{n}_h)'\right)_{L^2(\omega)}\\
      =\ & - \frac{\partial \bar{\mathcal E}_{\operatorname{res},h}(\hat{n}_h^{k-1})}{\partial \hat{n}_h^{k-1}}[\delta \hat{n}_h] - \frac{\partial N_h(y_h^{k},b_h^{k},\hat{n}_h^{k-1})}{\partial \hat{n}_h^{k-1}}[\delta \hat{n}_h]
    \end{align*}
    for all $ \delta \hat{n}_h \in \mathcal F_{h,\hat{n}}(\hat{n}_h^{k-1}) $.
  \item Stop the iteration if $ \|d_ty_h^k\|_Y + \|d_tb_h^k\|_X + \|d_t\hat{n}_h^k\|_Z \leq \eps_{\operatorname{stop}} $. Otherwise set $ k \mapsto k+1 $ and continue with (1). 
  \end{enumerate}
\end{algorithm}

\subsection{Numerical experiments}
The experiments we present serve the purpose of investigating properties of the LCE-model and the proposed algorithm. Stability and convergence can be investigated following the scheme presented in \cite{BarRei20}, where these results are available. Additionally, our experiments indicate stability, at least for the parameters specified below.

We simulate an elastic rod made of a nearly incompressible isotropic material by using the Lam\'e parameters $ \lambda = 1000 $ and $ \mu = 1 $ and assume it to have a circular cross-section where the LCE-material fills one semi-circle, i.e.
\begin{align*}
  S = B(0;\pi^{-\frac{1}{2}}),\qquad S_0 = S \cap \{\bar{x} = (x_2,x_3) : x_3 > 0\}.
\end{align*}
With these specifications and Lemma \ref{L:isotropic}, we are able to infer the representations of $ \bar{Q} $ and $ K_{\operatorname{pre}} $. For $ E_{\operatorname{res}} $, the Lemmas \ref{L:coeffrep} and \ref{L:correctors1} imply that we need to solve several quadratic minimization problems to assemble the matrix $ \mathbb{E}_{\operatorname{res}} $. The corresponding linear elliptic systems are approximately solved using a standard finite-element-method and lead to
\begin{align*}
  \mathbb{E}_{\operatorname{res}} = 10^{-2}
  \begin{pmatrix}
    0.95 & 0 & 0 & 0 & 0\\
    0 & 10.77 & -0.01 & 0 & 0\\
    0 & -0.01 & 0.18 & 0 & 0\\
    0 & 0 & 0 & 34.94 & 0\\
    0 & 0 & 0 & 0 & 4.9
  \end{pmatrix}.
\end{align*}
This matrix characterizes $ E_{\operatorname{res}} $ with regard to the basis $ \{U_1,\dots,U_5\} $ of $ \mathbb{R}^{3\times 3}_{\operatorname{dev}} $ which is used in Lemma \ref{L:coeffrep}.

Additionally, we choose the spacial step size $ h = 1/200 $ and the constant timestep size $ \tau = h/2 = 1/400 $ as well as the model parameter $ \eps = 1 / 200 $. For $ y_h, \bar{y}_h \in \mathcal{S}^{3,1}(\mathcal{T}_h)^3 $ and $ b_h, \bar{b}_h \in \mathcal{S}^{1,0}(\mathcal{T}_h)^3 $, the bilinear forms we use are given by
\begin{align*}
  (y_h , \bar{y}_h)_Y &= (y_h , \bar{y}_h)_{L^2(\omega)} + h(y_h'' , \bar{y}_h'')_{L^2(\omega)},\\
  (b_h , \bar{b}_h)_X = (b_h , \bar{b}_h)_Z &= (b_h , \bar{b}_h)_{L^2(\omega)} + h(b_h' , \bar{b}_h')_{L^2(\omega)}.
\end{align*}
We next specify boundary conditions and external forces which we use in our experiments.

% of $ \hat{n}_h^k $ we add the source term $ \left( f , R_h^{k} m_h \right) $ to the right-hand side to simulate a homogeneous magnetic field that forces the lc-director $ n = R\hat{n} $ in the direction of $ f \in \real^3 $. We can change $ f $ after a certain time to manipulate the deformation. \\
    %     The boundary conditions we use are fixed on one side and free on the other.\\
    %     In the pictures, the different coloured arrows represent the different variables, where $ y' $ is blue, $ b $ is red, $ d = y' \times b $ is black and $ n = R\hat{n} $ is green.\\
    %     
\example[Bending via magnetic field]\label{ex_magn_field}
Our first experiment focuses on a straight line from the clamped end $ (0,0,0) $ to the free end $ (2,0,0) $. In the beginning $ y' = (1,0,0) $, $ b = (0,1,0) $, $ d = (0,0,1) $ and  $ n = \hat{n} = (0,1,0) $ are constant. For a visualisation of the starting configurations, see the first graphic in \autoref{figure_line_arrows}.
    %     \begin{figure}
    %     \centering
    %     \includegraphics[width=11cm]{pics/lce_test_line_arrows_0}
    %     
    %     \caption{Starting configuration for the first experiment. The different coloured arrows represent different variables, where $ y' $ is blue, $ b $ is red, $ d = y' \times b $ is black and $ n = R\hat{n} $ is green. In this case, $ b $ is not visible because $ b = n $.}
    %     \label{figure_line_start}
    %     \end{figure}

To simulate a homogeneous magnetic field that forces the LCE-director $ n $ to align with a vector $ f \in \R^3 $, we add the forcing term $ -\left( f , R \hat{n} \right)_{L^2(\omega)} $ to the energy. The computations of $ y $, $ b $ and $ \hat{n} $ are modified accordingly. We split the (quasi-) time interval $ [0,T] $ with $ T = 60 $ into the smaller intervals $ I_m = (10(m-1),10m] $ for $ m = 1,\dots,6 $. Since we are interested in the LCE-director's influence on bending behaviour, we choose the external field to change periodically between two constant states given by $ f_{\operatorname{odd}} = (0,1,0) $ and $ f_{\operatorname{even}} = (1,0,0) $. For $ t\in (0,T) $ and $ x_1 \in \omega $ we thus define
\begin{equation*}
  f(t,x_1) := \begin{cases}
    f_{\operatorname{odd}} & \text{if } t \in I_m, m \text{ is odd},\\
    f_{\operatorname{even}} & \text{if } t \in I_m, m \text{ is even}.
  \end{cases}
\end{equation*}
The remaining parameters we use are $ \bar{r} = \frankconstant{} = 1 $.
\begin{figure}
  \centering
  \includegraphics[width=4.7cm]{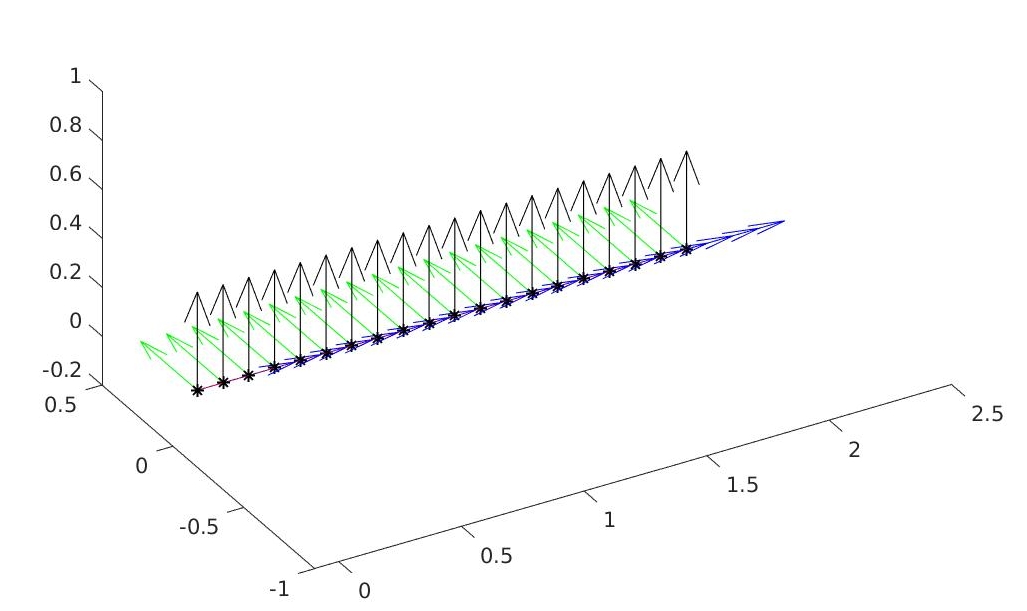}
  \includegraphics[width=4.7cm]{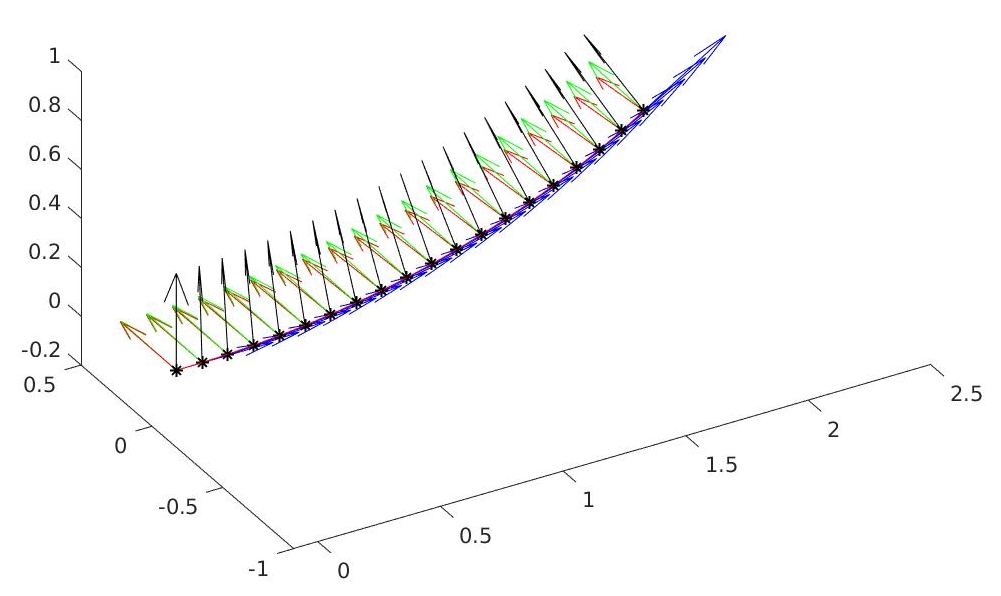}
  \includegraphics[width=4.7cm]{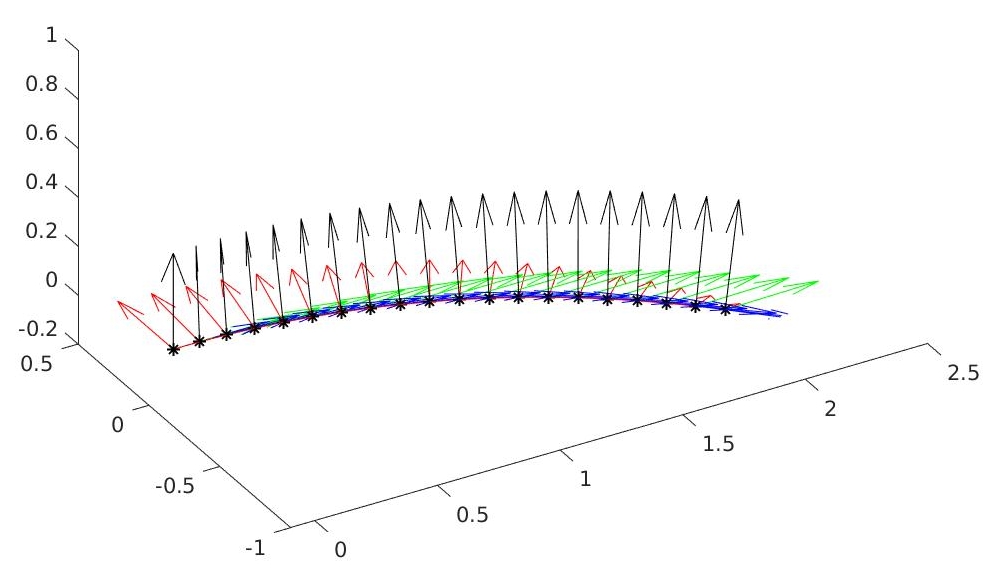}
  \includegraphics[width=1.3cm]{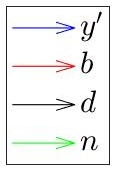}
  \caption{Deformation of the rod in Example~\ref{ex_magn_field}. From left to right we have $ t = 0 $, $ t = 30 $ and $ t = 40 $. The different coloured arrows represent the vectors $ y' $, $ b $ and $ d = y' \wedge b $, that form the frame $ R $, as well as the LCE-director $ n = R\hat{n} $. The two rightmost graphics show, that the LCE-director $ n $ tends to align with the vectors $ 	f_{\operatorname{odd}} = (0,1,0) $ or $ 	f_{\operatorname{even}} = (1,0,0) $ at the end of the intervals $ I_3 $ and $ I_4 $ respectively.}
  \label{figure_line_arrows}
\end{figure}
\begin{figure}
  \centering
  \includegraphics[width=12cm]{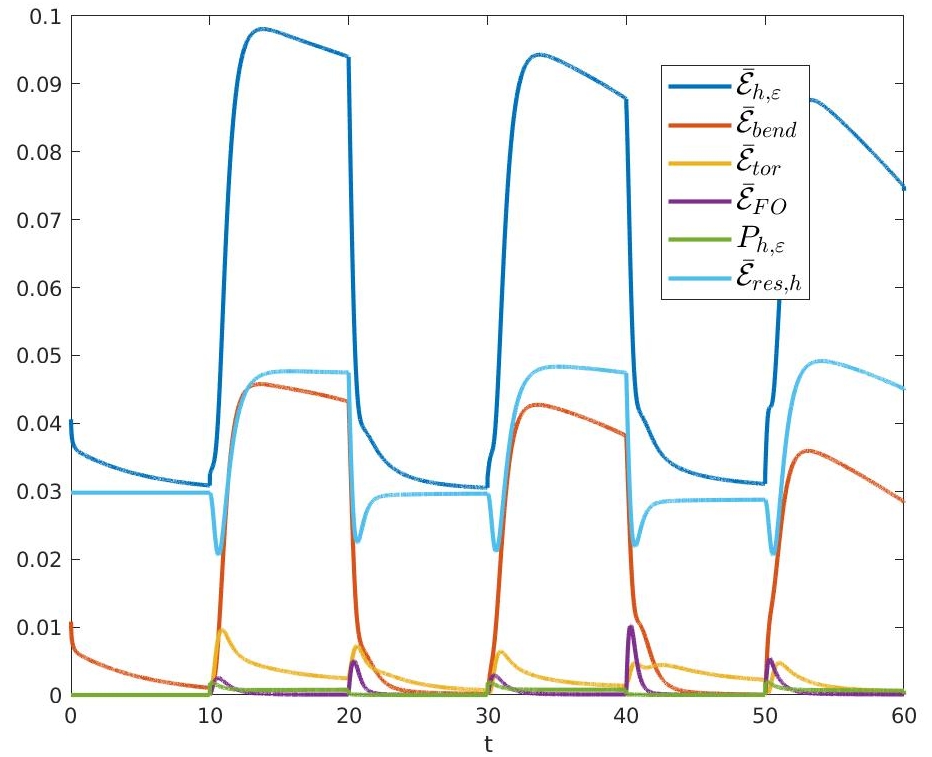}
  \caption{Energy development over the (quasi-) time interval in Example~\ref{ex_magn_field}. The different curves represent the energy term $ \bar{\mathcal{E}}_{h,\eps} $, the bending energy, the torsion energy, the Frank-Oseen energy, the penalty term $ P_{h,\eps} $ and the contribution of $ \bar{\mathcal E}_{\operatorname{res},h} $.	The energy is non-decreasing in some parts. This is due to the fact that the energy functional $ \bar{\mathcal E}_{h,\eps} $ does not include the forcing term $ -\left( f , R \hat{n} \right)_{L^2(\omega)} $, which was added to simulate a magnetic field, and it takes some time until $ n $ and $ f $ are realigned after $ f $ changes.}
  \label{figure_line_energy}
\end{figure}

\begin{figure}
  \centering
  \includegraphics[width=7cm]{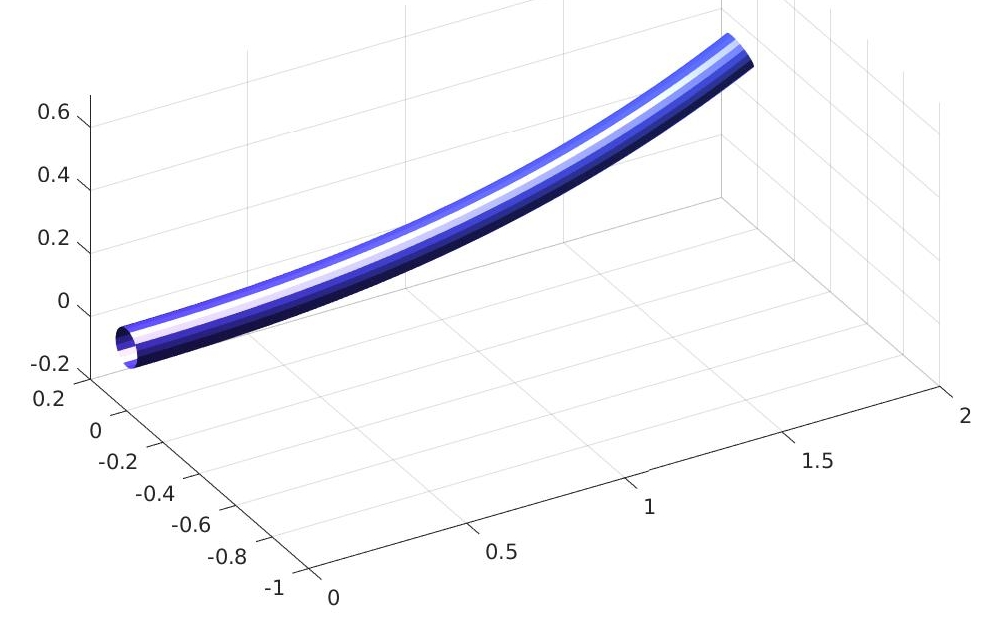}
  \includegraphics[width=7cm]{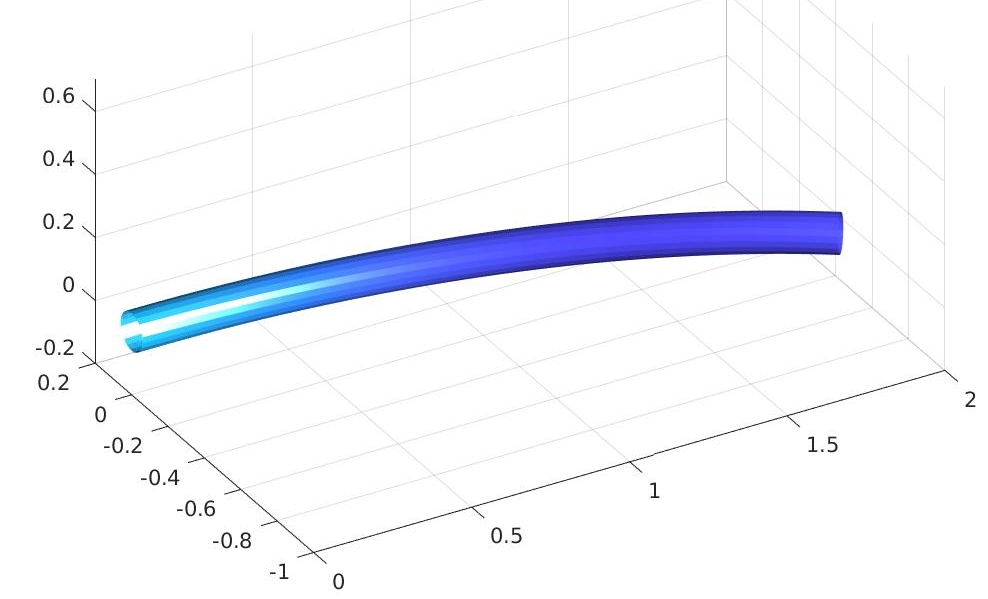}
  \includegraphics[width=7cm]{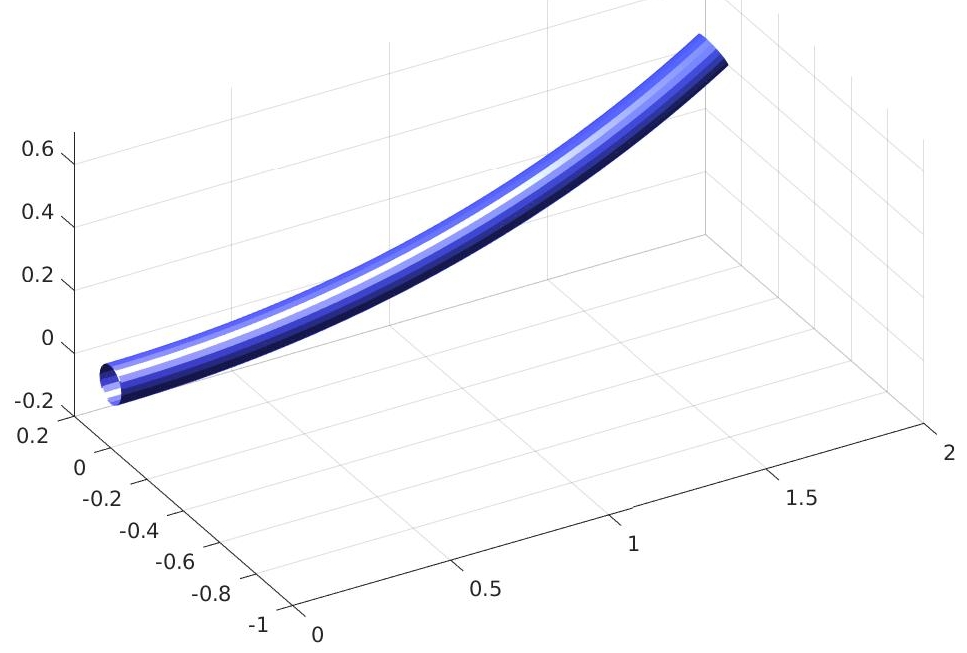}
  \includegraphics[width=7cm]{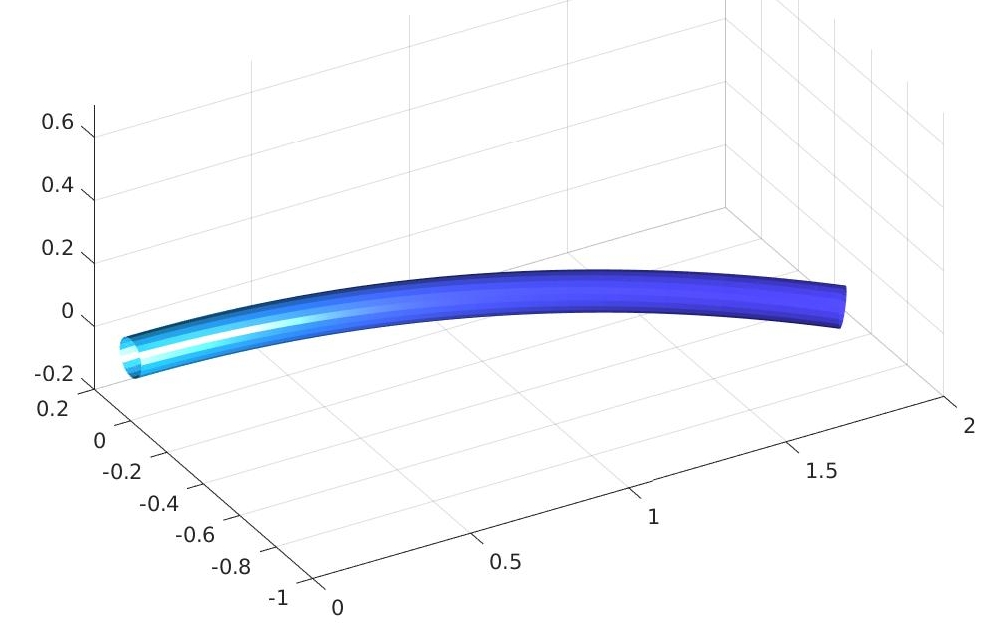}
  \caption{Deformation of the rod from Example~\ref{ex_magn_field}. The top row shows the configurations at $ t = 10 $ and $ t = 20 $, while the bottom row shows $ t = 30 $ and $ t = 40 $. The colouring of the tube represents its curvature.}
  \label{figure_line_development}
\end{figure}

In the Figures \ref{figure_line_energy} and \ref{figure_line_development} we can see the development of energy and deformation within $ (0,T) $. For both values of $ f $ we observe a different deformation the rod seems to converge to, basically enabling us to switch between two states. However, the deformations and energy at the end of the different time intervals corresponding to the same value of $ f $ are slightly different. This can be explained by the fact, that these intervals are too short for a full relaxation and convergence to the minimizer. Indeed we observe smaller differences when using longer time intervals.

\begin{figure}[t]
  \centering
  \includegraphics[width=10cm]{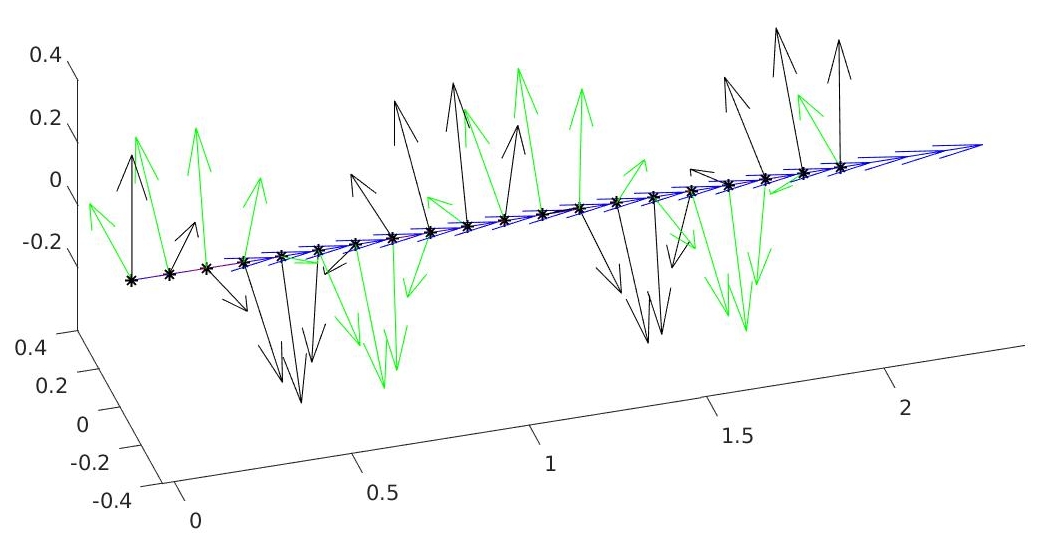}
  \includegraphics[width=1.5cm]{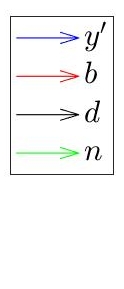}
  \includegraphics[width=7cm]{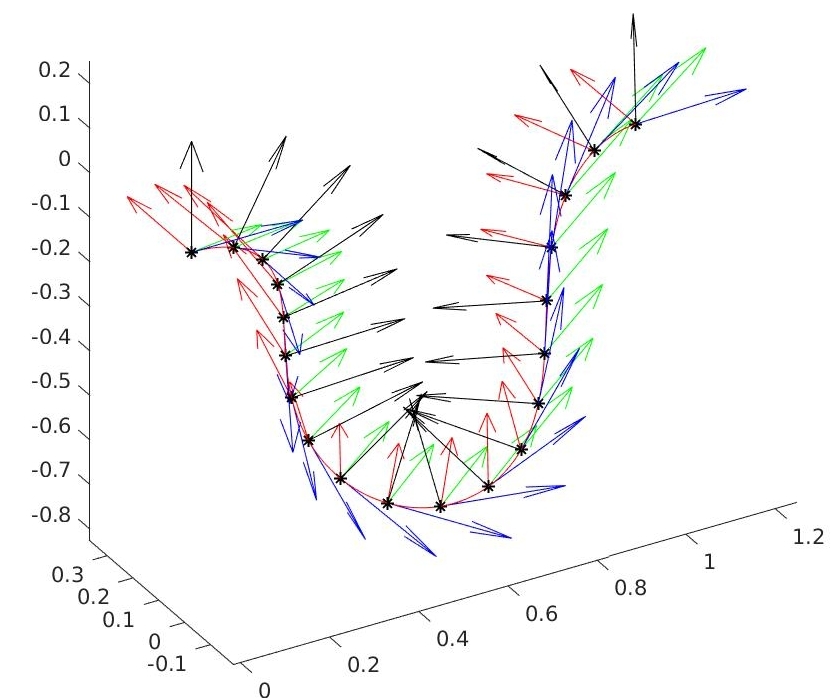}
  \includegraphics[width=7cm]{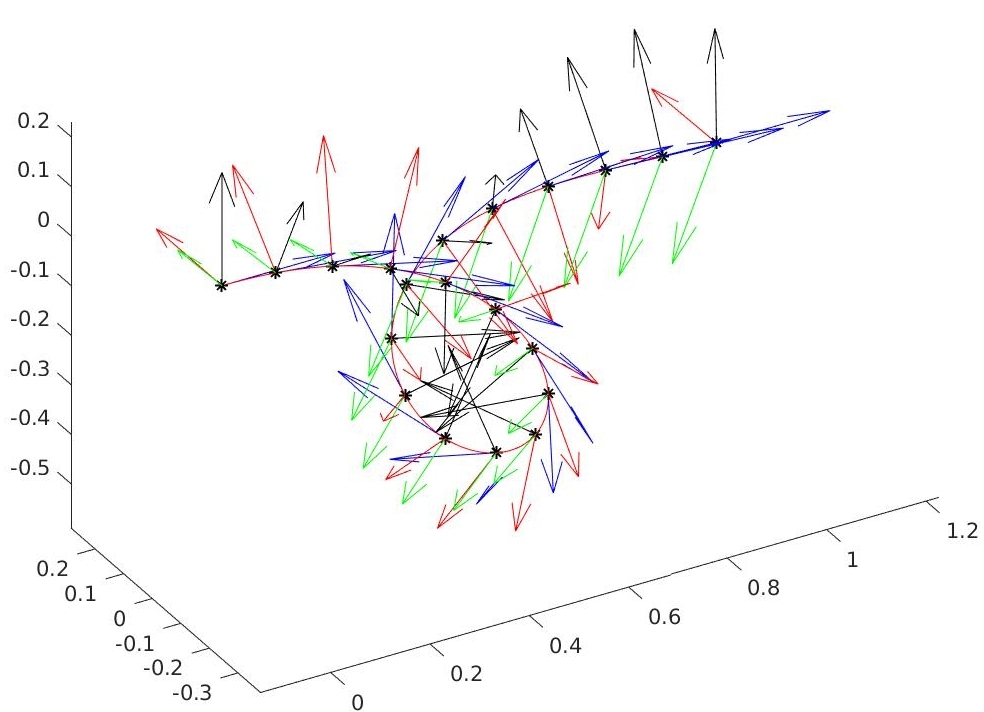}
  \caption{Buckling behaviour of the rod from Example~\ref{ex_buckling}. The top row shows the starting configuration while the bottom row depicts the relaxed state at time $ t = 50 $ with $ \bar{r} = 1 $ (left) and $ \bar{r} = 3 $ (right). As in \autoref{figure_line_arrows}, the vectors $ y' $, $ b $ and $ d = y' \wedge b $, that form the frame $ R $, as well as the LCE-director $ n = R\hat{n} $ are shown.}
  \label{figure_buckling_arrows}
\end{figure}
\example[Buckling]\label{ex_buckling}
In our second experiment, we aim to investigate the buckling behaviour of an LCE-rod. Again, we have a straight line from $ (0,0,0) $ to $ (2,0,0) $. This time however, both ends are clamped and the rod is twisted twice. The initial LCE-director is $ n = R\hat{n} = b $. This configuration can be seen in the first graphic of \autoref{figure_buckling_arrows}.
    %     \begin{figure}[h]
    %     \centering
    %     \includegraphics[width=11cm]{pics/lce_test_buckling_arrows_start.jpg}
    %     
    %     \caption{Starting configuration for the second experiment. The different coloured arrows represent different variables, where $ y' $ is blue, $ b $ is red, $ d = y' \times b $ is black and $ n = R\hat{n} $ is green. In this case, $ b $ is not visible because $ b = n $.}
    %     \label{figure_buckling_start}
    %     \end{figure}

In order to induce buckling behaviour, we modify the boundary conditions to move the end at $ (2,0,0) $ towards $ (0,0,0) $ with a velocity of $ (-1,0,0) $. At $ t = 1 $, when the end on the right-hand side is located at $ (1,0,0) $, we stop this movement and let the rod relax until $ T = 50 $. Since we are interested in the influence of the parameter $ \bar{r} $, we fix $\frankconstant{} = 0.4 $ and carry out the experiment for $ r = 1,\dots,10 $.
\begin{figure}
  \centering
  \includegraphics[width=11cm]{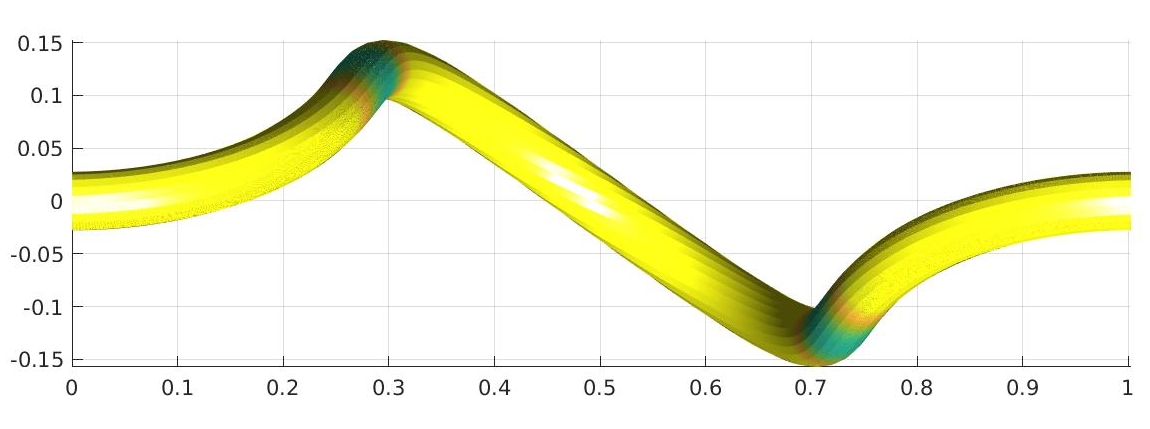}
  \includegraphics[width=11cm]{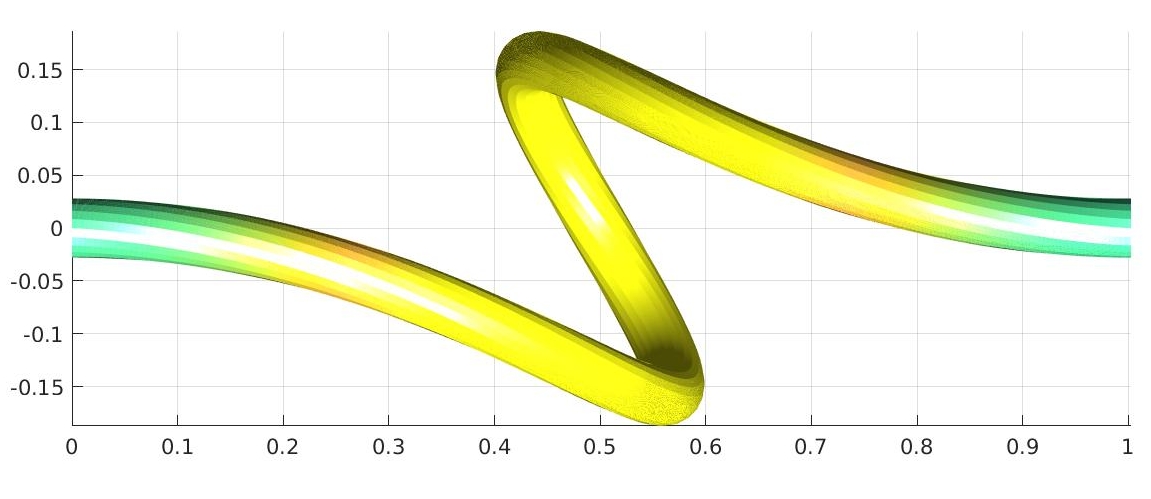}
  \includegraphics[width=11cm]{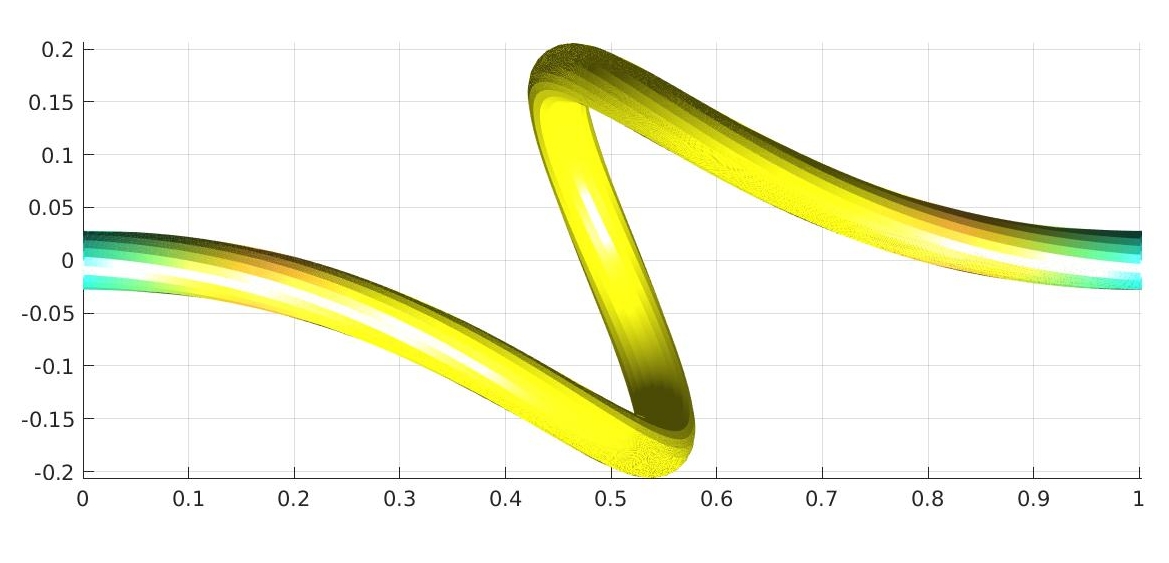}
  \caption{Buckling of the LCE-rod in Experiment~\ref{ex_buckling} with $ \bar{r} = 1,3,5 $ (top to bottom). The line of view is parallel to the $ x_3 $-axis and the colouring of the tube represents its curvature.}
  \label{figure_buckling_end}
\end{figure}

\begin{figure}
  \centering
  \includegraphics[width=7.5cm]{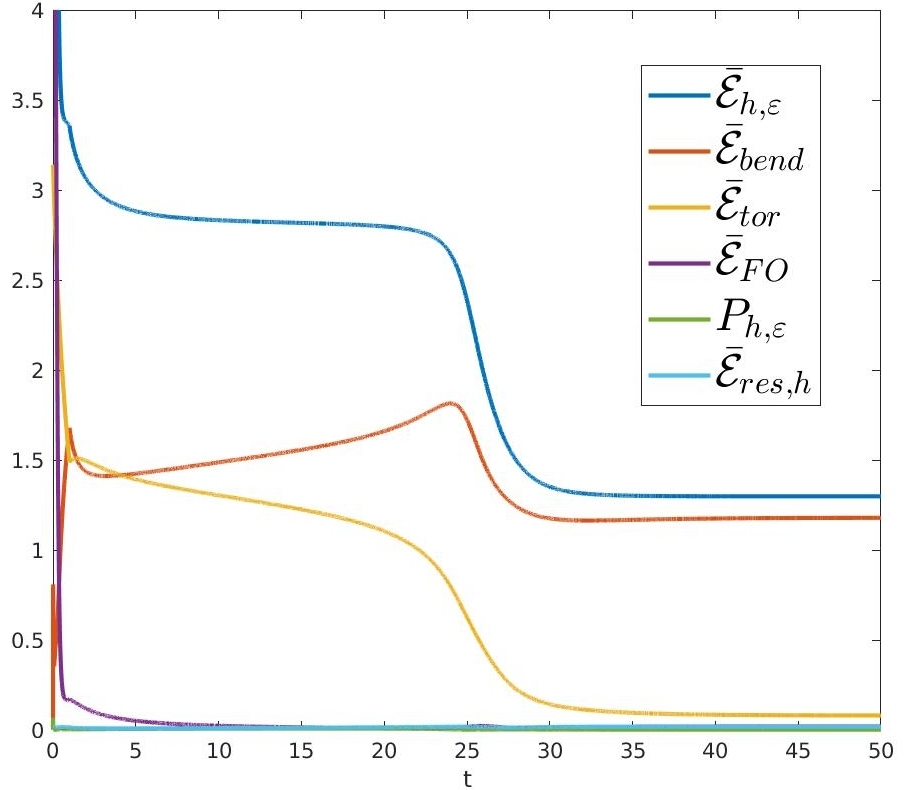}
  \includegraphics[width=7.5cm]{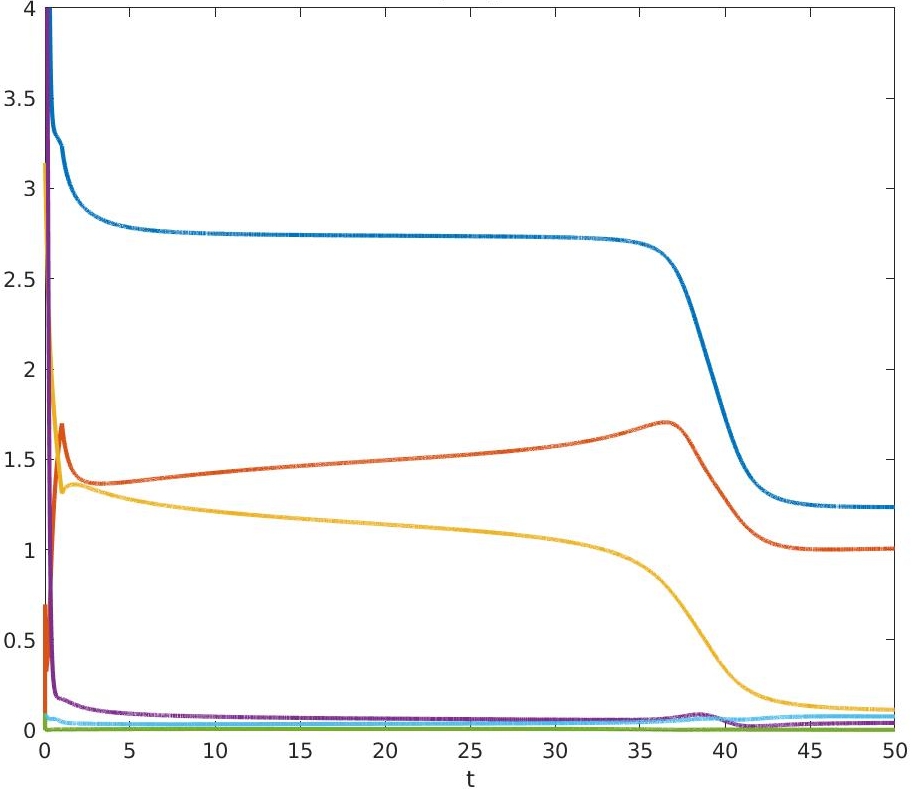}
  \includegraphics[width=7.5cm]{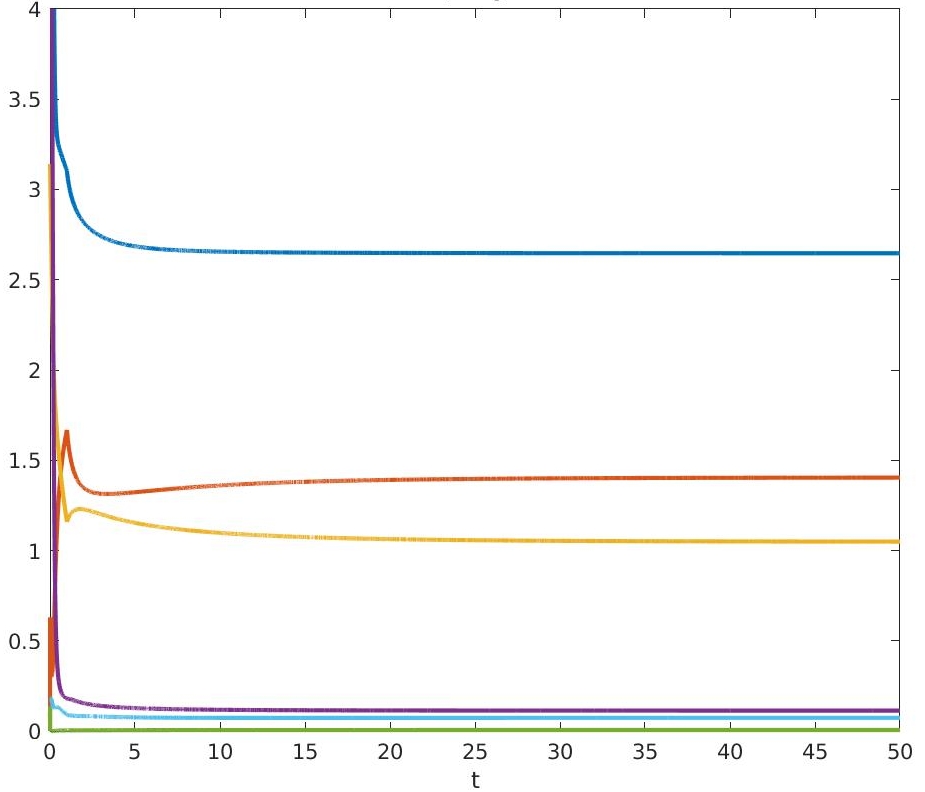}
  \includegraphics[width=7.5cm]{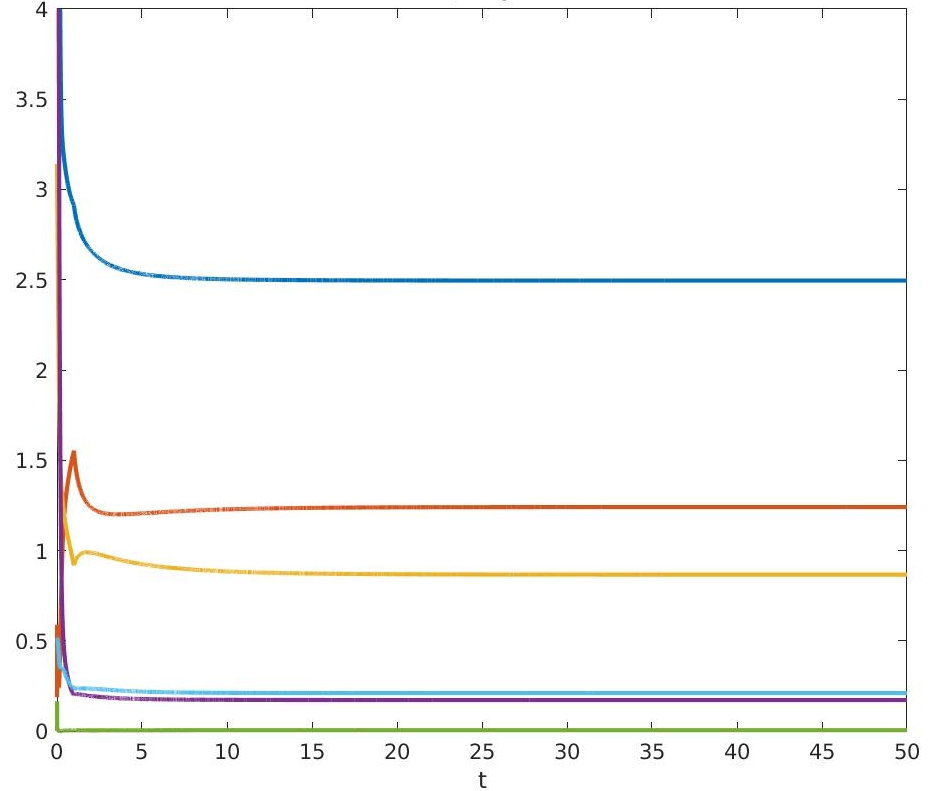}
  \caption{Energy development in Example~\ref{ex_buckling}. The top row shows the experiments with $ \bar{r} = 1 $ and $ \bar{r} = 2 $, while the bottom row shows $ \bar{r} = 3 $ and $ \bar{r} = 5 $. Again we see the total energy $ \bar{\mathcal{E}}_{h,\eps} $, the bending energy, the torsion energy, the Frank-Oseen energy, the penalty term $ P_{h,\eps} $ and the contribution of $ \bar{\mathcal E}_{\operatorname{res},h} $.}
  \label{figure_buckling_energy}
\end{figure}
Figures \ref{figure_buckling_arrows} and \ref{figure_buckling_end} show the deformation at the end of the time frame for various values of $ \bar{r} $. We observe a significant difference between the choice $ \bar{r} = 1 $ and the other cases. In order to understand the dependence, we take a look at the energy plots which can be seen in \autoref{figure_buckling_energy}. The cases $ \bar{r} = 1 $ and $ \bar{r} = 2 $ have a steep decline of mostly the torsion energy at some point after the curve initially flattens. In the other cases, this decline is not present within the given time frame. (Indeed, even when choosing a significantly longer time frame, the decline does not happen for $ \bar{r} \geq 3 $.)

\begin{figure}
  \centering
  \includegraphics[width=11cm]{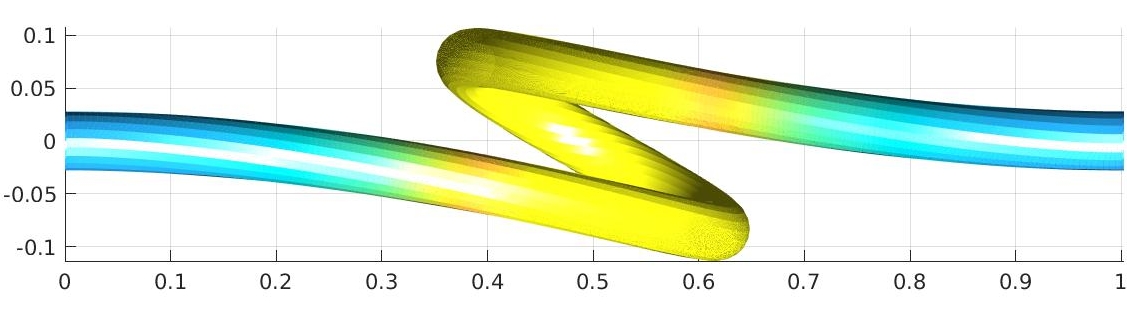}
  \includegraphics[width=11cm]{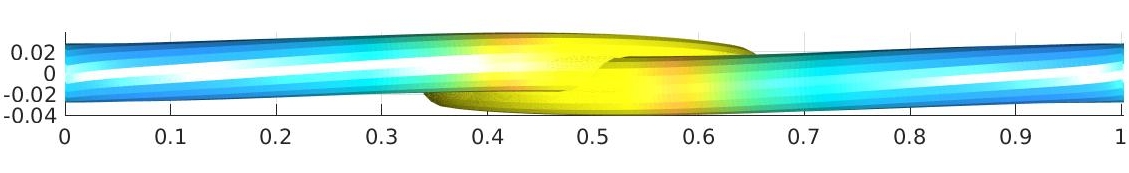}
  \includegraphics[width=11cm]{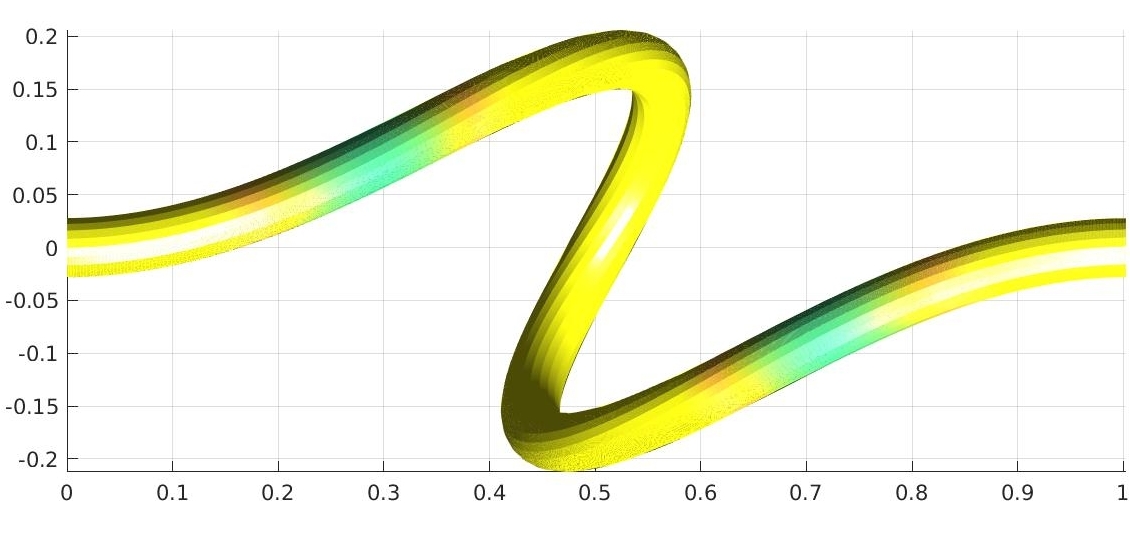}
  \caption{Buckling behaviour of the LCE-rod from Example~\ref{ex_buckling}. We have $ \bar{r} = 1 $ and see the configurations at $ t = 22,24,26 $ (top to bottom). The rod intersects itself to reach a more beneficial energetic state.}
  \label{figure_buckling_dropoff}
\end{figure}
The reason for the aforementioned decline is shown in \autoref{figure_buckling_dropoff}. The rod reaches a more beneficial energetic state by effectively folding over. In doing so, it intersects itself, which is possible due to the lack of a self-avoidance term in our model, but not concerning when we are just interested in minimizing the energy. In reality, the rod would likely be stopped by itself and form a loop, but the behavioural differences based on the parameter $ \bar{r} $ would still be significant.

In these two examples, we have seen that the competing contributions of the different energy terms in the LCE rod model result in non-trivial deformations. Additionally, the LCE material has a significant impact on the configuration that minimizes the energy. On the one hand, we are able to change the bending behaviour of a rod by directly influencing the LCE-director via a simulated magnetic field. On the other hand, we see a change in buckling behaviour when the model parameter $ \bar{r} $ is modified. In practice, this could be realized by altering the surrounding temperature.
\newpage

\section{Proofs}\label{S:proof}
\subsection{Compactness -- Proof of Theorem~\ref{T1} (a) and Proposition~\ref{P1} (a)}
In this step we prove \eqref{eq:comp1}, \eqref{eq:comp2}. Furthermore, we show that there exists $R_h:\Omega\to\SO 3$ and a corrector $    \chi\in L^2(\omega; \mathbb H_{\rm rel})$
such that (for a subsequence) we have
\begin{equation}\label{eq:comp3}
  \frac{R_h^\top L_h(n_h)^{-\frac12}\nabla_hy_h-I}{h}\wto \sym\Bigg((R^\top\partial_1R\barx)
  \otimes e_1\Bigg)+\tfrac{\bar r}{2}\chi_{\Omega_0}R^\top(\tfrac13 I-n\otimes n)R+\chi.
\end{equation}
Indeed, by the argument of  \cite[Lemma 4.1]{bartels2022nonlinear} we see that $\limsup\limits_{h\to 0}\mathcal E_h(y_h,n_h)<\infty$ implies
\begin{align}
  \label{eq:st:2d}
  &\limsup\limits_{h\to 0}\frac1{h^2}\int_{\Omega}\dist^2(\nabla\!_hy_h,\SO 3)<\infty,\\
  \label{eq:st:2c}
  &\limsup\limits_{h\to 0}\int_{\Omega}|\nabla\!_hy_h|^{q}+|\det(\nabla\!_hy_h)|^{-\frac{q}{2}}<\infty,\\
  \label{eq:st:2b}
  &\limsup\limits_{h\to0}\int_{\Omega_0}|\nabla\!_hn_h(\nabla\!_hy_h)^{-1}|^2|\det(\nabla\!_h y_h)|<\infty,\\\label{eq:st:2a}
  &\limsup\limits_{h\to0}\int_{\Omega_0}|\nabla\!_hn_h|^p<\infty,\qquad\text{where }p=\frac{2q}{q+4},
\end{align}
and we find $\bar h>0$ (depending on the sequence) such that
\begin{equation}\label{eq:st:0}
  \det(\nabla\!_hy_h)>0\text{ a.e.~in }\Omega\qquad\text{for all }0<h\leq\bar h.
\end{equation}
From now on we assume that $0<h\leq\bar h$.

Thanks to \eqref{eq:st:2a} we find $n\in H^1(\omega;\mathbb S^2)$ such that (up to a subsequence) we have $n_h\wto n$ weakly in $W^{1,p}(\Omega_0)$ and pointwise a.e.~in $\Omega_0$. Since $(n_h)$ is bounded in $L^\infty$, \eqref{eq:comp2} follows.
Furthermore, by appealing to \cite[Proposition 5.1 (a)]{bauer2019derivation} we conclude that there exists a subsequence (not relabeled), and a rod configuration $(y,R)$ such that \eqref{eq:comp1} holds and 
\begin{equation}\label{eq:comp5}
  E_h(y_h)\wto E:=\sym\Bigg((R^\top\partial_1R\barx)
  \otimes e_1\Bigg)+\chi\qquad\text{weakly in }L^2,
\end{equation}
for a corrector $\chi\in L^2(\omega; \mathbf H_{\rm rel})$.
To prove \eqref{eq:comp3} we proceed as follows: In view of \eqref{eq:st:0} and thanks to the polar decomposition we find $R_h:\Omega\to\SO 3$ such that
$$\nabla_hy_h(x)=R_h(x)\sqrt{\nabla_hy_h(x)^\top\nabla_hy_h(x)}.$$
We set $G_h:=\frac{R_h^\top\nabla_hy-I}{h}$ and note that \eqref{eq:comp5} implies $G_h\wto E$ weakly in $L^2$ and $R_h\to R$ strongly in $L^2$.
Next, we rewrite the elastic part of the strain as follows,
\begin{equation*}
  \Big(I+\chi_{\Omega_0}(L_h(n_h)^{-\frac{1}{2}}-I)\Big)\nabla_hy_h=R_h(I+h\widetilde B_h)(I+h G_h),\qquad \widetilde B_h:=\chi_{\Omega_0} R_h^\top \big(\frac{L_h(n_h)^{-\frac{1}{2}}-I}{h}\big) R_h.
\end{equation*}
In view of the construction of $L_h^{-\frac12}$, the boundedness of $n_h$, and the pointwise convergence of $n_h$ and $R_h$ a.e.~in $\Omega_0$, we deduce that $(\widetilde B_h)$ is bounded in $L^\infty(\Omega)$, and
\begin{equation*}
  \widetilde B_h\to \tfrac{\bar r}{2}\chi_{\Omega_0}R^\top(\tfrac13 I-n\otimes n)R\qquad\text{in }L^2(\Omega).
\end{equation*}
In combination with $G_h\wto E$, \eqref{eq:comp3} follows.
\qed

\subsection{Lower bound -- Proof of Theorem~\ref{T1} (b) }
By passing to a subsequence (not relabeled) and by appealing to compactness in form of Theorem~\ref{T1} (a), we may assume without loss of generality that \eqref{eq:comp3} holds, and
\begin{equation*}
  \liminf\limits_{h\to 0}\mathcal E_h(y_h,n_h)=    \lim\limits_{h\to 0}\mathcal E_h(y_h,n_h)<\infty.
\end{equation*}
From \eqref{eq:comp3} we infer with the argument of \cite[Proof of Theorem~2.3(b)]{bartels2022nonlinear} that
\begin{align*}
  &\liminf\limits_{h\to 0}\frac{1}{h^2|S|}\Big(\int_{\Omega\setminus\Omega_0}W\Big(\nabla_hy_h(x)\Big)\,dx+\int_{\Omega_0}W\Big(L_h(n_h(x))^{-\frac12}\nabla_hy_h(x)\Big)\,dx\Big)\\
  \geq & \frac{1}{|S|}\int_{\Omega}Q\Big(\sym\Big((R^\top\partial_1 R\,\barx)
         \otimes e_1\Big)+\tfrac{\bar r}{2}\chi_{\Omega_0}R^\top(\tfrac13 I-n\otimes n)R+\chi\Big)\,dx\\
  \geq & \int_{\omega}\min_{\chi\in {\Hrel}}\fint_{S}Q\Big(\sym\Big((R^\top \partial_1R\,\barx)
         \otimes e_1\Big)+\tfrac{\bar r}{2}\chi_{S_0}R^\top(\tfrac13 I-n\otimes n)R+\chi\Big)\,d\bar x\,dx_1.
\end{align*}
We combine this with Lemma~\ref{L:relax} to deduce
\begin{equation}\label{eq:LB:1}
  \begin{aligned}
    &\liminf\limits_{h\to 0}\frac{1}{h^2|S|}\Big(\int_{\Omega\setminus\Omega_0}W\Big(\nabla_hy_h(x)\Big)\,dx+\int_{\Omega_0}W\Big(L_h(n_h(x))^{-\frac12}\nabla_hy_h(x)\Big)\,dx\Big)\\
    \geq & \int_{\omega}\bar Q\Big(R^\top \partial_1R+\bar r K_{\pre}\big(\tfrac13 I-R^\top n\otimes R^\top n\big)\Big)+\bar r^2 E_{\res}\big(\tfrac13 I-R^\top n\otimes R^\top n\big)\,dx_1.
  \end{aligned}
\end{equation}
It remains to treat the Frank-Oseen energy. By \eqref{eq:comp1} and \eqref{eq:comp2} we have $(\nabla_hy_h,n_h)\to (R,n)$ in $L^2$. Furthermore, by the a propri bound \eqref{eq:st:2a} we have (up to a subsequence and for some $p\in(1,2)$) $\nabla_h n_h\to (\partial_1n,d)$ in $L^p$ where $d\in L^p(\Omega_0;\R^{3\times 2})$. As in the proof of \cite[Theorem~2.3 (a)]{bartels2022nonlinear} we conclude that $\nabla_hn_h\nabla_hy_h^{-1}(\det\nabla_hy_h)^{\frac12}\wto (\partial_1n,d)R^{\top}$ weakly in $L^p$, and thus, by lower semicontinuity, we get
\begin{eqnarray*}
  \liminf\limits_{h\to 0}\frac{1}{|S_0|}\int_{\Omega_0}|\nabla_hn_h\nabla_hy_h^{-1}|^2\det\nabla_hy_h\,dx&\geq& \int_\omega\fint_{S_0}|(\partial_1n,d)R^{\top}|^2\,d\bar x\,dx_1\\
                                                                                                       &\geq&\int_\omega|\partial_1n|^2\,dx_1,
\end{eqnarray*}
where for the last step we used $|(\partial_1n,d)R^{\top}|^2=|(\partial_1n,d)|^2\geq |\partial_1n|^2$. In combination with \eqref{eq:LB:1} the claimed lower bound follows.
\qed

\subsection{Upper bound -- Proofs of Theorem~\ref{T1} (c) and Proposition~\ref{P1} (b)}\label{S:proof:UB}
The statements of Theorem~\ref{T1} (c) and Proposition~\ref{P1} (b) directly follow from the following stronger statement (which we shall also use in the proof of Proposition~\ref{P2}): For all $(y,R,n)\in\mathcal A$ and $\beta\in[0,1)$ there exists a sequence $(y_h)\subset C^\infty(\bar\Omega;\R^3)$ and $n_h\in H^1(\Omega_0;\mathbb S^2)$ such that
\begin{equation}
  \label{eq:UB:eq0}
  y_h(0,\bar x)=y(0)+hR(0)\barx,\qquad 
\end{equation}
and
\begin{equation}\label{eq:diagonal}
  \begin{aligned}
    \limsup\limits_{h\to 0}&\Big(\Big|\mathcal E(y,R,n)-\mathcal E_h(y_h,n_h)\Big|+\|y_h-y\|_{L^2}+\|n_h-n\|_{H^1}\\
    &\qquad +h^{-\beta}\|\nabla_hy_h-R\|_{L^\infty} + h^{-\beta}\Big\|\tfrac{(\nabla_hy_h)^\top n_h}{|(\nabla_hy_h)^\top n_h|}-R^\top n\Big\|_{L^\infty}=0.
  \end{aligned}
\end{equation}
For the proof we proceed in two steps.

\step 1 (Approximation argument).

We first argue that for each $k\in\N$ we can choose $(y^k,R^k,n^k)\in\mathcal A$ such that
$y^k,R^k\in C^\infty(\overline\omega)$,
\begin{equation}\label{eq:UB:eq1}
  y^k(0)=y(0),\qquad R^k(0)=R(0),\qquad (R^k)^\top n^k=R^\top n,
\end{equation}
and
\begin{equation}\label{eq:UB:close}
  \Big|\mathcal E(y,R,n)-\mathcal E(y^k,R^k,n^k)\Big|+\|y-y^k\|_{H^2}+\|R-R^k\|_{L^\infty}+\|\partial_1R-\partial_1R^k\|_{L^2}+\|n-n^k\|_{H^1}<\frac1{k}.
\end{equation}
Indeed, in view of the approximation result \cite[Lemma~2.3]{neukamm2012rigorous} for all $\delta>0$ there exists a smooth framed curve $(y^k,R^k)$ satisfying the one-sided boundary condition \eqref{eq:UB:eq1} such that $\|y^k-y\|_{H^2}+\|R^k-R\|_{L^\infty}+\|\partial_1R^k-\partial_1R\|_{L^2}<\delta$. Set $n^k:=R^kR^\top n$. Then \eqref{eq:UB:eq1} is satisfied and $\|\partial_1n^k-\partial_1 n\|_{L^2}\leq \delta(\|\partial_1n\|_{L^2}+1)$. Since $\mathcal E$ is continuous, \eqref{eq:UB:close} follows by choosing $\delta$ sufficiently small.

Set $K^k:=(R^k)^\top\partial_1 R^k$ and $U^k=\frac13 I-(R^k)^\top n^k\otimes (R^k)^\top n^k$. By Lemma~\ref{L:relax} and an approximation argument we find $a^k\in C^\infty_c(\omega)$ and $\varphi^k\in C^\infty_c(\omega; C^\infty(\bar S;\R^3))$ such that
\begin{equation}\label{eq:UB:eq2}
  \begin{aligned}
    &\Bigg|\mathcal E(y^k,R^k,n^k) -\frac{1}{|S|}\int_\Omega Q\Big(\bar x,\sym\big(K^k\barx \otimes e_1+\tfrac{\bar r}{2}  \boldsymbol{1}_{S_0}U^k+(a^ke_1,\bar\nabla\varphi^k)\big)\Big)\,dx\\
    &\qquad -\frankconstant{}^2\int_\omega|\partial_1n^k|^2\,dx_1\Bigg|<\frac1{k}.
  \end{aligned}
\end{equation}

\step 2 (Definition of the recovery sequence).

Let $y^k,R^k,a^k,\varphi^k$ be as in Step~1, and define for $h>0$ the smooth 3d-deformation
\begin{equation*}
  y^k_h(x):=y^k(x_1)+hR^k(x_1)\barx +h\int_0^{x_1}a^k(s)R^k(s)e_1\,ds+h^2R^k(x_1)\varphi^k(x).
\end{equation*}
A direct calculation yields   
\begin{equation}
  \begin{aligned}
    \frac{R_k^\top\nabla_hy_h^k-I}{h}=\,&K^k\barx \otimes e_1+ (a^ke_1,\bar\nabla\varphi^k) +hO^k_h,\\
    \frac{R_k^\top   \Big(I+\chi_{\Omega_0}(L_h(n^k)^{-\frac{1}{2}}-I)\Big)\nabla_hy_h^k-I}{h}=&\,K^k\barx \otimes e_1+ (a^ke_1,\bar\nabla\varphi^k)\\
    &+\tfrac{\bar r}{2}\chi_{\Omega_0}(\tfrac13 I-(R^k)^\top n^k\otimes (R^k)^\top n^k+o^k_h,
  \end{aligned}
\end{equation}
with remainders $O^k_h, o^k_h$ satisfying $\limsup_{h\to 0}\|O^k_h\|_{L^\infty}<\infty$ and $\lim_{h\to 0}\|o^k_h\|_{L^\infty}=0$.
Hence, by frame-indifference and an expansion at identity we get
\begin{align*}
  &\lim\limits_{h\to 0}\Big(\frac{1}{h^2|S|}\int_{\Omega\setminus\Omega_0}W\Big(\bar x,\nabla_hy_h^k\Big)\,dx+\frac{1}{h^2|S|}\int_{\Omega_0}W\Big(\bar x,L_h(n^k)^{-\frac12}\nabla_hy^k_h(x)\Big)\,dx\Big)\\
  =&\lim\limits_{h\to 0}\Big(\frac{1}{h^2|S|}\int_\Omega W\Big(\bar x,\big(I+\chi_{\Omega_0}(L_h(n^k)^{-\frac{1}{2}}-I)\big)\nabla_hy_h^k\Big)\,dx\\
  = & \frac{1}{|S|}\int_\Omega Q\Big(\bar x,\sym\big(K^k\barx \otimes e_1+\tfrac{\bar r}{2}  \boldsymbol{1}_{S_0}U^k+(a^ke_1,\bar\nabla\varphi^k)\big)\Big)\,dx,
\end{align*}
and
\begin{align*}
  &\lim\limits_{h\to 0}\frac{\frankconstant{}^2}{|S_0|}\int_{\Omega_0}\big|\nabla_hn^k(\nabla_hy_h^k)^{-1}\big|^2\det(\nabla_hy_h^k)\,dx=\frankconstant{}^2\int_\omega|\partial_1n^k|^2\,dx_1.
\end{align*}
Together with \eqref{eq:UB:close} and \eqref{eq:UB:eq2} we thus conclude that for any $\beta\in[0,1)$,
\begin{equation*}
  \begin{aligned}
    \limsup\limits_{k\to\infty}\limsup\limits_{h\to 0}&\Big(\Big|\mathcal E(y,R,n)-\mathcal E_h(y^k_h,n^k)\Big|+\|y^k_h-y\|_{L^2}+\|n^k-n\|_{H^1}\\
    &\qquad +h^{-\beta}\|\nabla_hy^k_h-R\|_{L^\infty} + h^{-\beta}\Big\|\tfrac{(\nabla_hy_h^k)^\top n^k}{|(\nabla_hy_h^k)^\top n^k|}-R^\top n\Big\|_{L^\infty}=0
  \end{aligned}
\end{equation*}
We thus obtain the sought for sequence $(y_h,n_h)$  by extracting a diagonal sequence.
\qed

\subsection{Anchoring -- Proof of Proposition~\ref{P2}}
\step 1 (Proof of (a)).

Since $\mathcal G_h$ is non-negative, the conclusion of Theorem~\ref{T1} (a) on compactness remains valid for the modified functional $\mathcal E_h+\mathcal G_h$.

Next we prove the lower-bound: Consider a sequence $(y_h,n_h)\subset H^1(\Omega;\R^3)\times H^1(\Omega_0;\mathbb S^2)$ and $(y,R,n)\in\mathcal A$ such that $(y_h,\nabla_hy_h,n_h)\to (y,R,n)$ stronlgy in $L^2$. In view of Theorem~\ref{T1} (b), in order to conclude the claimed lower bound $\liminf\limits_{h\to 0}\mathcal E_h(y_h,n_h)+\mathcal G_h(y_h,n_h;\hat n_{bc})\geq \mathcal E(y,R,n)+\mathcal G_{\textrm{weak}}(y,R,n;\hat n_{bc})$, we only need to show that
\begin{equation}\label{eq:anch:1}
  \liminf\limits_{h\to 0}\mathcal G_h(y_h,n_h;\hat n_{bc})\geq \mathcal G_{\textrm{weak}}(y,R,n;\hat n_{bc}).
\end{equation}
W.l.o.g.~we may assume that $\limsup\limits_{h\to 0}\mathcal G_h(y_h,n_h;\hat n_{bc})<\infty$. Hence, $f_h:=\left|\frac{\nabla_hy_h^\top n_h}{|\nabla_hy_h^\top n_h|}-\hat n_{bc}  \right|_a\sqrt{\rho}$ defines a bounded sequence in $L^2$.
Since $(\nabla_hy_h,n_h)\to (R,n)$ in $L^2$ and thus pointwise a.e.~(up to extraction of a subsequence), we conclude that $f_h\wto \left|R^\top n-\hat n_{bc}  \right|_a\sqrt{\rho}$ weakly in $L^2$. Thus, \eqref{eq:anch:1} follows by lower semicontinuity of the convex integral functional $f\mapsto \fint_{\Omega_0}|f|^2_a\rho\,dx$ and the definition of $\mathcal G_{\textrm{weak}}$.

We finally prove the upper-bound. To that end let $(y,R,n)\in\mathcal A$ and denote by $(y_h,n_h)\subset H^1(\Omega;\R^3)\times H^1(\Omega_0;\mathbb S^2)$ the recovery sequence constructed in Section~\ref{S:proof:UB} satisfying \eqref{eq:diagonal} for some $\beta\in(0,1)$. Moreover, we note that there exists a constant $C>0$ (independent of $h$) such that
\begin{equation}\label{eq:anch:2}
  \begin{aligned}
    &\,\Big|\mathcal G_h(y_h,n_h;\hat n_{bc})-\mathcal G_{\textrm{weak}}(y,R,n;\hat n_{bc})\Big|\\
    =&\,\Big| \int_{\Omega_0}\Big(\left|\frac{\nabla_hy_h^\top n_h}{|\nabla_hy_h^\top n_h|}-\hat n_{bc}  \right|_a^2-|R^Tn-\hat n_{bc}|^2_a\Big)\rho\,dx\Big|\\
    \leq&\,\left(\int_{\Omega_0}\Big(\left|\frac{\nabla_hy_h^\top n_h}{|\nabla_hy_h^\top n_h|}-R^Tn\right|^2_a\Big)\rho\,dx\right)^\frac12\left(\int_{\Omega_0}\Big(\left|\frac{\nabla_hy_h^\top n_h}{|\nabla_hy_h^\top n_h|}+R^Tn-2\hat n_{bc}\right|^2_a\Big)\rho\,dx\right)^\frac12\\
    \leq&\, C\left\|\tfrac{\nabla_hy_h^\top n_h}{|\nabla_hy_h^\top n_h|}-\hat n_{bc}\right\|_{L^\infty(\Omega_0)}.
  \end{aligned}
\end{equation}
Hence, \eqref{eq:diagonal} implies that $\lim_{h\to 0}\mathcal E_h(y_h,n_h)+\mathcal G_h(y_h,n_h;\hat n_{bc})=\mathcal E(y,R,n)+\mathcal G_{\textrm{weak}}(y,R,n;\hat n_{bc})$ as claimed.

\step 2 (Proof of (b)).

Since $\mathcal G_h$ is non-negative, the conclusion of Theorem~\ref{T1} (a) on compactness remains valid for the modified functional $\mathcal E_h+h^{-\beta}\mathcal G_h$.

Next we prove the lower bound. To that end consider a sequence $(y_h,n_h)\subset H^1(\Omega;\R^3)\times H^1(\Omega_0;\mathbb S^2)$ and suppose that $(y_h,\nabla_hy_h,n_h)\to (y,R,n)$ strongly in $L^2$ for some $(y,R,n)\in\mathcal A$. W.lo.g.~we may assume that $\limsup\limits_{h\to 0}\mathcal E_h(y_h,n_h)+h^{-\beta}\mathcal G_h(y_h,n_h;\hat n_{bc})<\infty$. Note that this implies that $\mathcal G_h(y_h,n_h;\hat n_{bc})\to 0$. Combined with \eqref{eq:anch:1} we infer that
\begin{equation*}
  0=\mathcal G_{\textrm{weak}}(y,R,n;\hat n_{bc})=\int_\omega|R^\top n-\hat n_{bc}|_a^2\bar\rho\,dx_1,
\end{equation*}
and thus $\mathcal G_{\textrm{strong}}(y,R,n;\hat n_{bc})=0$. Hence, Theorem~\ref{T1} (b) yields the claimed lower bound
\begin{equation*}
  \liminf\limits_{h\to 0}\mathcal E_h(y_h,n_h)+h^{-\beta}\mathcal G_h(y_h,n_h;\hat n_{bc})\geq \mathcal E(y,R,n)=\mathcal E(y,R,n)+\mathcal G_{\textrm{strong}}(y,R,n;\hat n_{bc}).
\end{equation*}
We remark that the constructed recovery sequence additionally satisfies the clamped one-sided boundary condition \eqref{eq:UB:eq0}.

For the upper bound let $(y,R,n)\in\mathcal A$. It suffices to construct the recovery sequence in the case $\mathcal G_{\textrm{strong}}(y,R,n;\hat n_{bc})=0$. To that end denote by $(y_h,n_h)\subset H^1(\Omega;\R^3)\times H^1(\Omega_0;\mathbb S^2)$ the recovery sequence of Section~\ref{S:proof:UB} satisfying \eqref{eq:diagonal}. Then
\begin{equation*}
  |h^{-\beta}\mathcal G_h(y_h,n_h;\hat n_{bc})|=  h^{-\beta}|\mathcal G_h(y_h,n_h;\hat n_{bc})-\mathcal G_{\textrm{weak}}(y,R,n;\hat n_{bc})|\leq Ch^{-\beta}\left\|\tfrac{\nabla_hy_h^\top n_h}{|\nabla_hy_h^\top n_h|}-\hat n_{bc}\right\|_{L^\infty(\Omega_0)},
\end{equation*}
and by \eqref{eq:diagonal} we conclude that the right-hand side converges to $0$ as $h\to 0$. Hence,
\begin{equation*}
  \lim\limits_{h\to 0}\big(\mathcal E_h(y_h,n_h)+h^{-\beta}\mathcal G_h(y_h,n_h;\hat n_{bc})\big)=\mathcal E(y,R,n)=\mathcal E(y,R,n)+\mathcal G_{\textrm{strong}}(y,R,n;\hat n_{bc}).
\end{equation*}
We remark that the constructed recovery sequence additionally satisfies the clamped one-sided boundary condition \eqref{eq:UB:eq0}.

\subsection{Proofs of Lemmas~\ref{L:relax}, \ref{L:coeffrep}, \ref{L:correctors1}, and \ref{L:isotropic}}

\begin{proof}[Proof of Lemma~\ref{L:relax}]By appealing to $(\cdot,\cdot)_Q$-orthogonality and the definition of the orthogonal projections $P_{\Hmicro}, P_{\Hres}, P_{\Hmacro}$, we conclude that
  \begin{eqnarray*}
    &&\inf_{\chi\in{\Hrel}}\fint_SQ(K\bar x\otimes e_1+\tfrac{\bar r}{2}  \boldsymbol{1}_{S_0}U+\chi)\\
    &=&\min_{\chi\in{\Hrel}}\|K\bar x\otimes e_1+P_{\Hmicro}\big(\tfrac{\bar r}{2}  \boldsymbol{1}_{S_0}U\big)+\chi\|^2_Q+\|P_{\Hres}\big(\tfrac{\bar r}{2}  \boldsymbol{1}_{S_0}U\big)\|^2_Q\\
    &=&\|P_{\Hmacro}\big(K\bar x\otimes e_1+\tfrac{\bar r}{2}  \boldsymbol{1}_{S_0}U\big)\|^2_Q+\bar r^2E_{\res}(U)\\
    &=&\|\mathbf E(K+\bar rK_{\pre}(U))\big)\|^2_Q+\bar r^2E_{\res}(U)\\
    &=&\bar Q\big(K+\bar rK_{\pre}(U)\big)+\bar r^2E_{\res}(U).
  \end{eqnarray*}

\end{proof}

\begin{proof}[Proof of Lemma~\ref{L:coeffrep}]
  To prove  \eqref{eq:formQ}, note that by linearity of $\mathbf E$ and the definition of $\Psi_i$ we have $\mathbf E(K)=\sum_{i=1}^3k_i\Psi_i$, and thus
  \begin{equation*}
    \bar Q(K)=\big\|\sum_{i=1}^3k_i\Psi_i\big\|_Q^2=\sum_{i,j=1}^3k_ik_j\big(\Psi_i,\Psi_j\big)_Q=k\cdot\mathbb Qk.
  \end{equation*}

  Next, we prove \eqref{eq:Kpre}. From the definition of $\mathbb Q$ and $\mathbb U$ we conclude that
  \begin{equation*}
    \big(P_{\Hmacro}(  \boldsymbol{1}_{S_0}U_j),\Psi_{\hat i}\Big)_Q=\big(\boldsymbol{1}_{S_0}U_j,\Psi_{\hat i}\Big)_Q=\sum_{i=1}^3\big(\mathbb Q^{-1}\mathbb U\big)_{ij}\Big(\Psi_i,\Psi_{\hat i}\Big)_Q\qquad\text{for }\hat i=1,2,3.
  \end{equation*}
  Hence, since $\{\Psi_1,\Psi_2,\Psi_3\}$ is a basis of $\Hmacro$, we get
  \begin{equation}\label{eq:Pmacroident}
    P_{\Hmacro}(\boldsymbol{1}_{S_0}U_j)=\sum_{i=1}^3\big(\mathbb Q^{-1}\mathbb U\big)_{ij}\Psi_i.
  \end{equation}
  In view of the definition of $K_{\pre}$, an application of $\mathbf E^{-1}$ to both sides yields
  \begin{equation*}
    K_{\pre}(U_j)=\frac12\sum_{i=1}^3\big(\mathbb Q^{-1}\mathbb U\big)_{ij}\mathbf E^{-1}\big(\Psi_i\big).
  \end{equation*}
  Combined with the definition of $\Psi_i$, which yields $\mathbf E^{-1}\big(\Psi_i\big)=K_i$, we conclude $K_{\pre}(U_j)=\tfrac12\sum_{i=1}^3\big(\mathbb Q^{-1}\mathbb U\big)_{ij}K_i$ and \eqref{eq:Kpre} follows by linearity of $K_{\pre}$.
  Finally, we note that in view of orthogonal decomposition $\mathbb H=\Hres\oplus\Hmacro\oplus\Hrel$ and \eqref{eq:Pmacroident}, we have $\Phi_j=P_{\Hres}(\boldsymbol 1_{S_0}U_j)$, and thus  \eqref{eq:Eres} follows.
\end{proof}

\begin{proof}[Proof of Lemma~\ref{L:correctors1}]
  By the variational characterization of $P_{\Hrel}$ we have for all $F\in\mathbb H$,
  \begin{equation*}
    P_{\Hrel}F=-\chi_F\text{ where }\chi_F\in\Hrel\text{ is characterized by }\|F+\chi_F\|_Q=\inf_{\chi\in\Hrel}\|F+\chi\|_Q.
  \end{equation*}
  In view of Lemma~\ref{L:Hrelrep} we further conclude that $\chi_F=(a_F,\bar\nabla\varphi_F)$, where $(a_F,\varphi_F)\in\R\times H^1_{\rm av}(S;\R^3)$ is the unique minimizer of the functional \eqref{eq:corr1}. Applied with $F=\boldsymbol 1_{S_0}U_j$, the claimed identity for $\Phi_j$ follows. The identity for $\Psi_i$ follows by considering $F=\sym(K_i\barx\otimes e_1)$. Indeed, by definition we have $\Psi_i=P_{\Hmacro}\big(\sym(K_i\barx\otimes e_1)\big)$ and since $\sym(K_i\barx\otimes e_1)\in\Hmicro$, we conclude (with help of the orthogonal decomposition $\Hmicro=\Hmacro\oplus\Hrel$) that $\Psi_j=\sym(K_i\barx\otimes e_1)-P_{\Hrel}\big(\sym(K_i\barx\otimes e_1)\big)$.
\end{proof}
\begin{proof}[Proof of Lemma~\ref{L:isotropic} (a)]
  \todoneukamm{Note that $Q(G)=\frac12\mathbb LG\cdot G$ where
    \protect\begin{equation*}
      \frac12\mathbb LG=\frac\lambda 2(\trace G)I_{3\times 3}+\mu\sym G.
      \protect\end{equation*}
  }
  \step 1 (Calculation of $\mathbb Q$).
  
  We first claim that $\fint_S\bar\nabla\varphi_i\,d\bar x=0$. To that end consider
  \begin{equation*}
    A:=\fint_S\bar\nabla\varphi_i\,d\bar x,\qquad \tilde\varphi(\bar x):=\varphi_i(\bar x)-A\bar x.
  \end{equation*}
  Since $Q$ is independent of $\bar x$ and $\fint_SK_i\barx\otimes e_1\,d\bar x=0$, we have
  \begin{equation*}
    \fint_SQ\big(K_i\barx\otimes e_1+(a_ie_1,\bar\nabla\varphi_i)\big)\,d\bar x=
    \fint_SQ\big(K_i\barx\otimes e_1+(0,\bar\nabla\tilde\varphi)\big)\,d\bar x+
    Q\big((ae_1,A)\big).
  \end{equation*}
  By minimality of $(a_i,\varphi_i)$ we conclude that $(a,A)=0$ and thus the claim follows.

  \textit{Substep 1.1} We prove the formula for $q_1$ and claim that
  \begin{equation*}
    \Psi_1 =\sym\left(
      \begin{pmatrix}
        0\\
        \tfrac{x_3}{\sqrt 2}+\partial_2\alpha_S\\
        -\tfrac{x_2}{\sqrt 2}+\partial_3\alpha_S
      \end{pmatrix}\otimes e_1\right).
  \end{equation*}
  We already now that $a_1=0$ and $\fint_S\bar\nabla\varphi_1\,d\bar x=0$.
  For the argument write $\varphi_1=(\alpha,\bar\varphi)^\top$ and $\mathcal K_1:=\sym(K_i\barx\otimes e_1)$. Since $\trace \mathcal K_1=0$ a.e.~in $S$, a direct calculation yields
  \todoneukamm{
    \protect\begin{equation*}
      \mathcal K_1+\sym(0,\nabla\varphi)=\frac12
      \protect\begin{pmatrix}
        0&\frac{x_3}{\sqrt 2}+\partial_2\alpha&-\frac{x_2}{\sqrt 2}+\partial_3\alpha\\
        \frac{x_3}{\sqrt 2}+\partial_2\alpha&&\\
        -\frac{x_2}{\sqrt 2}+\partial_3\alpha&&
        \protect\end{pmatrix}+
      \protect\begin{pmatrix}
        0&0_2\\
        0_2&\sym\bar\nabla\bar\varphi
        \protect\end{pmatrix}
      \protect\end{equation*}
  }
  \begin{equation*}
    \|\mathcal K_1+\sym(0,\bar\nabla\varphi_1)\|^2_Q=\fint_S\tfrac{\mu}{2}(|\partial_2\alpha+\tfrac{x_3}{\sqrt 2}|^2+|\partial_2\alpha-\tfrac{x_2}{\sqrt 2}|^2)+\tfrac{\lambda}{2}(\trace\bar\nabla\bar\varphi)^2+\mu|\sym\bar\nabla\bar\varphi|^2\,d\bar x.
  \end{equation*}
  Since $\varphi_1$ is a minimizer, we conclude that $\bar\varphi=0$ and that $\alpha$ minimizes \eqref{eq:alpha}. Hence, $\varphi_1=(\alpha_S,0,0)$. Now, the identities for $\Psi_1$ and $q_1$ follow by a direct calculation.
  \smallskip

  \textit{Substep 1.2} We prove the formulas for $q_2$ and $q_3$ and claim that
  \begin{align*}
    \Psi_2=\frac{x_2}{\sqrt 2}{\rm diag}\big(1,-\tfrac12\tfrac{\lambda}{\lambda+\mu},-\tfrac12\tfrac{\lambda}{\lambda+\mu}\big),\qquad  \Psi_3=\frac{x_3}{\sqrt 2}{\rm diag}\big(1,-\tfrac12\tfrac{\lambda}{\lambda+\mu},-\tfrac12\tfrac{\lambda}{\lambda+\mu}\big).
  \end{align*}
  Indeed, by appealing to the Euler-Lagrange equation  one can check that
  \begin{equation}\label{eq:varphi2}
    \varphi_2=-\frac1{4\sqrt 2}\frac{\lambda}{\lambda+\mu}
    \begin{pmatrix}0\\
      x_2^2-x_3^2\\2x_2x_3
    \end{pmatrix}+C_2,\qquad     \varphi_3=-\frac1{4\sqrt 2}\frac{\lambda}{\lambda+\mu}
    \begin{pmatrix}0\\
      2x_2x_3\\x_3^2-x_2^2
    \end{pmatrix}+C_3
  \end{equation}
  where $C_2,C_3\in\R^3$ are chosen such that $\fint_S\varphi_2=\fint_S\varphi_3=0$.  Now, the identities for $\Psi_i,q_i$, $i=2,3$ follow by direct calculations.
  \smallskip

  \textit{Substep 1.3} A direct calculation shows that $\Psi_1,\Psi_2,\Psi_3$ are $(\cdot,\cdot)_Q$-orthogonal. Hence, $\mathbb Q={\rm diag}(q_1,q_2,q_3)$ as claimed.

  \step 2 (Formulas for $K_{\pre}$).
  
  \todoneukamm{Nachfolgendes muss ich nochmal ueberpreufen.}
  In the isotropic case we have for $i=1,2,3$ and all $U\in\R^{3\times 3}_{\dev}$,
  \begin{eqnarray*}
    \big(\boldsymbol{1}_{S_0}U,\Psi_i\big)_Q
    &=&\mu\fint_S  \boldsymbol{1}_{S_0}U\cdot\Psi_i\,d\bar x,
  \end{eqnarray*}
  since $\trace U=0$. Hence, a direct calculation yields
  \begin{equation*}
    \mathbb U=
    \begin{pmatrix}
      0& \mu\tfrac{\sqrt 2}2\fint_{S}\boldsymbol 1_{S_0}(\partial_2\alpha_S{+}\tfrac{x_3}{\sqrt 2})\,d\bar x& \mu\tfrac{\sqrt 2}2\fint_{S}\boldsymbol 1_{S_0}(\partial_3\alpha_S{-}\tfrac{x_2}{\sqrt 2})\,d\bar x &0&0\\
      0& 0 & 0 & \frac{\mu(3\lambda+2\mu)}{\lambda+\mu}\frac{1}{2\sqrt 3}\fint_S\boldsymbol 1_{S_0}x_2\,d\bar x&0\\
      0& 0 & 0 & \frac{\mu(3\lambda+2\mu)}{\lambda+\mu}\frac{1}{2\sqrt 3}\fint_S\boldsymbol 1_{S_0}x_3\,d\bar x&0
    \end{pmatrix},
  \end{equation*}
  and thus,
  \begin{equation*}
    \mathbb Q^{-1}\mathbb U=
    \begin{pmatrix}
      0& \frac{\sqrt{2}}{|S|c_S}\int_{S_0}\partial_2\alpha_S{+}\tfrac{x_3}{\sqrt 2}\,d\bar x& \frac{\sqrt{2}}{|S|c_S}\int_{S_0}\partial_3\alpha_S{-}\tfrac{x_2}{\sqrt 2}\,d\bar x &0&0\\
      0& 0 & 0 & \frac{2}{\sqrt 3}\frac{\int_{S_0}x_2\,d\bar x}{\int_Sx_2^2\,d\bar x}&0\\
      0& 0 & 0 & \frac{2}{\sqrt 3}\frac{\int_{S_0}x_3\,d\bar x}{\int_Sx_3^2\,d\bar x}&0
    \end{pmatrix}.
  \end{equation*}

\end{proof}

\begin{proof}[Proof of Lemma~\ref{L:isotropic} (b)]
  Since $S$ is a disc centered at $0$, the vector $(x_3,-x_2)^\top$ is tangential for all  $\bar x\in \partial S$ and a short calculation shows that
  \begin{equation*}
    \fint_S\Big|
    \begin{pmatrix}
      \partial_2\alpha_S\\\partial_3\alpha_S
    \end{pmatrix}
    +\frac{1}{\sqrt 2}
    \begin{pmatrix}
      x_3\\-x_2
    \end{pmatrix}\Big|^2\,d\bar x=\fint_S|\nabla\alpha_S|^2+    \frac12\fint_S|(x_3,-x_2)^\top|^2\,d\bar x.    
  \end{equation*}
  Since $\alpha_S$ is a minimizer, we conclude $\alpha_S=0$. Now the identities for $q_1,q_2,q_3$ and $K_{\pre}$ follow by direct calculations starting from the formulas derived in Lemma~\eqref{L:isotropic} (a). We note that for $K_{\pre}$ we also use that $\int_{S_0}x_2\,d\bar x=0$.
  
  \todoneukamm{Nebenrechnung:
    For $S=B(0;\pi^{-1/2})$ and $S_0=S\cap\{x_3>0\}$ have $|S|=1$, and
    \protect\begin{eqnarray*}
              c_S=\frac12\fint_{B(0;\pi^{-1/2})}|x|^2=\frac12\frac{|\partial B(0;1)|}{|B(0;\pi^{-1/2})|}\int_0^{-\pi^{-\frac12}}r^3\,dr=\frac1{4\pi}.
              \protect\end{eqnarray*}
            With    $r:=\pi^{-\frac12}$ get
            \protect\begin{eqnarray*}
                      \int_{\{x_3\geq 0\}\cap S}x_3\,d\bar x
                      &=&\int_{-r}^{r}\int_0^{\sqrt{r^2-x_2^2}}x_3\,dx_3\,dx_2=\frac12\int_{-r}^r(r^2-x_2^2)\,dx_2=r^3-\frac13 r^3=\frac23 r^3=\frac23\pi^{-3/2}.
                          \protect\end{eqnarray*}
                        Moreover,
                        \protect\begin{eqnarray*}
                                  \fint_{B(0;r)}x_3^2\,d\bar x=\frac12\fint_{B(0;r)}|\bar x|^2\,d\bar x=\frac1{2|B(0;r)|}\int_0^r|\partial B(0;\rho)|\rho^2\,d\rho=\frac{1}{r^2}\int_0^r\rho^3\,d\rho=\frac1 4 r^2 = \frac{1}{4\pi}
                                  \protect\end{eqnarray*}
                                Since $\int_{S_0}x_2\,dx_2=0$, get
                                \protect\begin{eqnarray*}
                                          K_{\pre}(U)&=\,&\,
                                                           \frac{4 u_2}{3\sqrt\pi}K_1+\frac{16 u_4}{3\sqrt{3\pi}}K_3\\      
                                          \protect\end{eqnarray*}
                                      }
                                    \end{proof}

                                  \subsection*{Acknowledgments}
                                  The authors acknowledge support by the German Research Foundation (DFG) via the
                                  re\-search unit FOR 3013 ``Vector- and tensor-valued surface PDEs'' (grant no.~BA2268/6--1 and \mbox{NE2138/4-1}).
                                  \bibliographystyle{alpha}
                                  \bibliography{bibliography.bib}
                                \end{document}